\documentclass[a4paper,12pt]{article}
\usepackage[text={6.5in,8.4in},centering]{geometry}
\usepackage{graphicx,url,subfig}
\RequirePackage[OT1]{fontenc}
\RequirePackage{amsthm,amsmath,amssymb,amscd}
\RequirePackage[numbers]{natbib}
\RequirePackage[colorlinks,citecolor=blue,urlcolor=blue]{hyperref}
\usepackage{enumerate}
\usepackage{datetime}
\usepackage{etex}
\usepackage{multirow}

\makeatother
\numberwithin{equation}{section}
\allowdisplaybreaks

\usepackage{pdflscape}
\usepackage{afterpage}

\theoremstyle{plain}
\newtheorem{theorem}{Theorem}[section]
\newtheorem{corollary}[theorem]{Corollary}
\newtheorem{lemma}[theorem]{Lemma}
\newtheorem{proposition}[theorem]{Proposition}

\theoremstyle{definition}
\newtheorem{definition}[theorem]{Definition}

\newtheorem{assumption}[theorem]{Assumption}

\newtheorem{remark}[theorem]{Remark}

\newtheorem{example}[theorem]{Example}

\newcommand{\E}{\mathbb{E}}
\newcommand{\W}{\dot{W}}
\newcommand{\ud}{\ensuremath{\mathrm{d}}}

\newcommand{\Floor}[1]{\left\lfloor #1 \right\rfloor}

\newcommand{\Indt}[1]{1_{\left\{#1 \right\}}}

\newcommand{\Norm}[1]{\left|\left|  #1   \right|\right|}

\newcommand{\Itos}{It\^{o}'s }

\newcommand{\InPrd}[1]{\left\langle #1 \right\rangle}

\newcommand{\Hcd}{$\widetilde{\text{\textnormal{H}}}$}
\newcommand{\Wcd}{$\widetilde{\text{\textnormal{W}}}$}

\newcommand{\calA}{\mathcal{A}}
\newcommand{\calB}{\mathcal{B}}
\newcommand{\calD}{\mathcal{D}}

\newcommand{\calF}{\mathcal{F}}
\newcommand{\calG}{\mathcal{G}}
\newcommand{\calK}{\mathcal{K}}
\newcommand{\calH}{\mathcal{H}}
\newcommand{\calL}{\mathcal{L}}
\newcommand{\calM}{\mathcal{M}}
\newcommand{\calN}{\mathcal{N}}

\newcommand{\calP}{\mathcal{P}}

\newcommand{\bbN}{\mathbb{N}}

\newcommand{\R}{\mathbb{R}}
\newcommand{\myEnd}{\hfill$\square$}

\newcommand{\He}{\ensuremath{\mathrm{He}}}

\newcommand{\conRef}[1]{(#1)}

\DeclareMathOperator{\Lip}{\mathit{L}}
\DeclareMathOperator{\LIP}{Lip}
\DeclareMathOperator{\lip}{\mathit{l}}

\DeclareMathOperator{\Vip}{\overline{\varsigma}}
\DeclareMathOperator{\vip}{\underline{\varsigma}}
\DeclareMathOperator{\vv}{\varsigma}

\def\sd{\beta}

\newcommand{\myRef}[2]{#1}

\title{{\bf  
Moment bounds in spde's with application to the stochastic wave equation}}

\author{{\bf Le Chen$^*$} and {\bf Robert C. Dalang\footnote{Research
partially supported by the Swiss National Foundation for Scientific
Research.}}\\
\\
\it\small Institut de math\'ematiques\\
\it\small \'Ecole Polytechnique F\'ed\'erale de Lausanne \\
\it\small Station 8 \\
\it\small CH-1015 Lausanne\\
\it\small Switzerland\\
\small \textit{e-mails:}
le.chen@epfl.ch, robert.dalang@epfl.ch\\
\small
\vspace{-1.5em}
\date{}
}

\begin{document}

\maketitle
\begin{center}
\begin{minipage}[rct]{5 in}
\footnotesize \textbf{Abstract:}
We exhibit a class of properties of an spde that guarantees existence,
uniqueness and bounds on moments of the solution. These moment bounds are
expressed in terms of quantities related to the associated deterministic
homogeneous p.d.e.
With these, we can, for instance, obtain solutions to the stochastic heat
equation on the real line for initial data that falls in a certain class of
Schwartz distributions, but our main focus is the stochastic wave equation on
the real line with irregular initial data. We give bounds on higher moments,
and for the hyperbolic Anderson model, explicit formulas for second moments.
We establish weak intermittency and obtain sharp bounds on exponential growth
indices for certain classes of initial conditions with unbounded support.
Finally, we relate H\"older-continuity properties of the
stochastic integral part of the solution to the stochastic
wave equation to integrability properties of the initial data, obtaining the
optimal H\"older exponent.

\vspace{2ex}

\vspace{2ex}
\vspace{4ex}
\end{minipage}
\end{center}


\section{Introduction}

Consider a partial differential operator $\calL$ in the time and space variables $(t,x)$ and a space-time white noise $\dot W(t,x)$, where $t \in \R_+^* = \R_+ \setminus \{0\}$ and $x \in \R^d$, along with a function $\theta(t,x)$. We are interested in determining when the stochastic partial differential equation (spde)
\begin{equation}\label{e1}
\calL u(t,x) = \rho\left(u(t,x)\right) \theta(t,x) \W(t,x)\;,\qquad
x\in\R^d, \;
t\in\R_+^*\;,
\end{equation}
with appropriate initial conditions, admits as solution a random field $(u(t,x),\ (t,x) \in \R_+ \times \R^d)$. In this case, we would like estimates and asymptotic properties of moments of $u(t,x)$, as well as H\"older-continuity properties. In this paper, we will develop such estimates for a wide class of operators $\calL$, functions $\theta$ and initial conditions, with an emphasis on the stochastic wave and heat equations.

   One basic example, which also was the starting point of this study, is the parabolic Anderson model. In this case, $d=1$, $\calL = \frac{\partial}{\partial t} - \kappa^2 \frac{\partial^2 }{\partial x^2}$, $\rho(x) = \lambda x$ and $\theta \equiv 1$. The intermittency property of this equation, as defined in \cite{CarmonaMolchanov94}, is studied via the moment Lyapounov exponents, in which estimates of the moments play a key role. Indeed, recall that the {\it upper and lower moment Lyapunov exponents} for constant initial data are defined as follows:
\begin{align}
 \overline{m}_p(x):=& \mathop{\lim\sup}_{t\rightarrow+\infty} \;
\frac{\log
\E\left[|u(t,x)|^p\right]}{t},\qquad
\underline{m}_p(x):=\mathop{\lim\inf}_{t\rightarrow+\infty} \;
\frac{\log
\E\left[|u(t,x)|^p\right]}{t}.
\end{align}
If the initial conditions are constants, then $\overline{m}_p(x)=: \overline{m}_p$ and $\underline{m}_p(x)=: \underline{m}_p$ do not depend on $x$. {\em Intermittency} is the property that $\underline{m}_p = \overline{m}_p =: m_p$ and $m_1 < m_2/2 < \cdots < m_p/p < \cdots$. It is implied by the property  $m_1=0$ and $\underline{m}_2>0$ (see \cite[Definition III.1.1, on p. 55]{CarmonaMolchanov94}), which is called {\em full intermittency}, while {\em weak intermittency}, defined in \cite{FoondunKhoshnevisan08Intermittence} and \cite[Theorem 2.3]{ConusEtal11WaveArxiv} is the property $\overline{m}_2>0$ and $\overline{m}_p < +\infty$, for all $p \geq 2$.

   Another property of the parabolic Anderson model is described by the behavior of exponential growth indices, initiated by Conus and Khoshnevisan in \cite{ConusEtal11WaveArxiv}. They defined
\begin{align}
\label{E4:GrowInd-0}
\underline{\lambda}(p):= &
\sup\left\{\alpha>0: \underset{t\rightarrow \infty}{\lim\sup}
\frac{1}{t}\sup_{|x|\ge \alpha t} \log \E\left(|u(t,x)|^p\right) >0
\right\},\\
\label{E4:GrowInd-1}
 \overline{\lambda}(p) := &
\inf\left\{\alpha>0: \underset{t\rightarrow \infty}{\lim\sup}
\frac{1}{t}\sup_{|x|\ge \alpha t} \log \E\left(|u(t,x)|^p\right) <0
\right\},
\end{align}
This is again a property of moments of the solution $u(t,x)$.

	In the recent paper \cite{ChenDalang13Heat}, in the case $\theta \equiv 1$,
the authors have given minimal conditions on the initial data for existence,
uniqueness and moments estimates in the parabolic Anderson model, building on
the previous results of \cite{BertiniCancrini94Intermittence,ConusEct12Initial}.
The initial condition can be a signed measure, but not a Schwartz distribution
that is not a measure, such as the derivative $\delta_0'$ of the Dirac delta
function. Exact formulas for the second moments were determined for the
parabolic Anderson model, along with sharp bounds for other moments and choices
of the function $\rho$.

	Our program is to extend these kinds of results to many other classes of
spde's. Recall that an spde such as \eqref{e1} is often rigorously formulated
as an integral equation of the form
\begin{align}\label{E:Int}
 u(t,x)=J_0(t,x)+ \iint_{\R_+\times\R^d} G(t-s,x-y) \rho(u(s,y)) \theta(s,y)
W(\ud s,\ud y),
\end{align}
where $J_0: \R_+\times\R^d$ represents the solution of the (deterministic)
homogeneous p.d.e. with the appropriate initial conditions, and $G(t,x)$ is the
fundamental solution of the p.d.e. The stochastic integral in \eqref{E:Int} is
defined in the sense of Walsh \cite{Walsh86}. In a first stage, we shall focus
on the equation \eqref{E:Int}, for given functions $J_0$ and $G$ satisfying
suitable assumptions, even if they are not specifically related to a partial
differential operator $\calL$. For this, the first step is to develop a unified
set of assumptions which are sufficient to guarantee the existence, uniqueness
and moment estimates
of the solution to \eqref{e1}. All of these assumptions should be satisfied for
the $J_0$ and $G$ associated with the stochastic heat equation, so as to
contain the results of \cite{ChenDalang13Heat}. It will turn out that in fact,
they can be verified for quite different equations, such as the stochastic wave
equation, which we discuss in this paper, and the stochastic heat equation with
fractional spatial derivatives as well as other equations, which will be
discussed in forthcoming papers.
	
	The assumptions are given in Section \ref{SS3:Assumption}. In particular, 
$G$ must be a function with certain continuity and integrability properties, and
must satisfy certain bounds, including tail control, and an $L^2$-continuity
property. Another assumption relates
properties of the function $J_0$ with those of $G$.  Finally, a last set of
assumptions concerns the function $\calK$ obtained by summing $n$-fold
space-time convolutions of the square of $G$ with itself.
	
	Our first theorem (Theorem \ref{T3:ExUni}) states that under these
assumptions, we obtain existence, uniqueness and moment bounds of the solution
to \eqref{E:Int}. When particularized to the stochastic heat equation, all the
assumptions are satisfied and the bounds are the same as those obtained in
\cite{ChenDalang13Heat}.
	
	Recall that $\theta(t,x) \equiv 1$ in \cite{ChenDalang13Heat}. Here, as an application of our first theorem, we will show in Theorem \ref{T3:ThetaHeat} that by choosing $\theta$ so that $\theta(t,x) \to 0$ as $t \downarrow 0$ (which means that we taper off the noise near $t=0$), we can extend the class of admissible initial conditions in the stochastic heat equation beyond signed measures. And the more the noise near the origin is killed, the more irregular the initial condition may be. The balance between the admissible
initial data and certain properties of the function $\theta$ is
stated in Theorem \ref{T3:ThetaHeat}. 
For instance, if $\theta(t,x)\equiv 1$, then the initial data cannot go beyond
measures; if
$\theta(t,x)=t^r \wedge 1$ for some $r>0$, then the initial data can be
$\delta^{(k)}_0$ for all integers $k\in [0,2r+1/2[\;$, where $\delta^{(k)}_0$
is the $k$-th distributional derivative of the Dirac delta function $\delta_0$;
if $\theta(t,x)=\exp\left(-1/t\right)$, then any Schwartz
(or tempered) distribution can serve as the initial data (see Examples \ref{e2.25} and \ref{e2.26}).
	
	The second and main application in this paper of our first theorem concerns the stochastic wave equation:
\begin{align}\label{E4:Wave}
\begin{cases}
\left(\frac{\partial^2 }{\partial t^2} - \kappa^2
\frac{\partial^2 }{\partial x^2}\right) u(t,x) =  \rho(u(t,x))
\:\dot{W}(t,x),&
x\in \R,\; t \in\R_+^*, \\
\quad u(0,\cdot) = g(\cdot), \;\; \frac{\partial u}{\partial t}(0,\cdot) =
\mu(\cdot),
\end{cases}
\end{align}
where $\R_+^*=\;]0,\infty[\;$, $\dot{W}$ is space-time white noise, $\rho(u)$ is
globally Lipschitz, $\kappa>0$ is the speed of wave propagation,
$g$ and $\mu$ are the (deterministic) initial position and velocity,
respectively.
The linear case, $\rho(u)=\lambda u$, $\lambda\ne 0$, is called
{\em the hyperbolic Anderson model} \cite{DalangMueller09Intermittency}.

   This equation has been intensively
studied during last two decades by many authors:
 see e.g.,
\cite{CairoliWalsh75SIPlane,
CarmonaNualart88Propagation,
CarmonaNualart88Smooth,
Orsingher82Randomly,
Walsh86} for some early work,
\cite{Dalang09Mini,Walsh86} for an introduction,
\cite{DalangMueller09Intermittency,
DalangMT06FKT}
for the intermittency problems,
\cite{
ConusDalang08Extending,
Dalang99Extending,
DalangFrangos98,
DalangQuer10Compare,
MilletMorien01On,
Peszat02Wave,
PeszatZabczyk97Ev}
for the stochastic wave equation in the spatial domain $\R^d$, $d>1$,
\cite{DalangSanzSole09Holder,
SanSoleSarra99Path}
for regularity of the solution,
\cite{
BrzezniakOndrejat07StrongWave,
BrzezniakOndrejat011Weak}
for the stochastic wave equation with values in
Riemannian manifolds,
\cite{
Chow02wave,
Ondrejat10WaveSobolev,
Ondrejat10WaveCritical}
for wave equations with polynomial nonlinearities,
and
\cite{
MilletSanzSol99SmoothLaw,
NualartQue07Existence,
QuerSanz04WaveSmooth}
for smoothness of the law.  

   Concerning intermittency properties, Dalang and Mueller showed in \cite{DalangMueller09Intermittency} that for the wave equation in spatial domain $\R^{3}$ with spatially
homogeneous colored noise, with $\rho(u) = u$ and constant initial position and velocity, the Lyapunov exponents
$\overline{m}_p$ and $\underline{m}_p$ are both bounded, from above and below respectively,  by some
constant times $p^{4/3}$.
For the stochastic wave equation in spatial dimension 1, Conus {\em et al}
\cite{ConusEtal11WaveArxiv} show that if the initial position and velocity are
bounded and measurable functions, then the moment Lyapunov exponents satisfy
$\overline{m}_p \leq C p^{3/2}$ for $p \geq 2$, and $\overline{m}_2 \geq c
(\kappa/2)^{1/2}$ for positive initial data. The difference in the
exponents---$3/2$ versus $4/3$ in the three dimensional wave equation---reflects
the distinct nature of the driving noises. Recently Conus and Balan
\cite{BalanConus13HeatWave} studied the problem when the noise is Gaussian,
spatially homogeneous and behaves in time like a fractional
Brownian motion with Hurst index $H>1/2$.

   Regarding exponential growth indices, Conus and Khoshnevisan \cite[Theorem
5.1]{ConusKhosh10Farthest} show that for initial data with exponential decay at
$\pm \infty$, $0<\underline{\lambda}(p)\le  \overline{\lambda}(p) <+\infty$, for
all $p \geq 2$. They also show that if the initial data consists of functions
with compact support, then $\underline{\lambda}(p) =
\overline{\lambda}(p)=\kappa$, for all $p \geq 2$.

   One objective of our study is to understand how irregular (and possibly unbounded) initial data affects the random field solutions to \eqref{E4:Wave}; another is to continue the study of moment Lyapounov exponents and exponential growth indices of \cite{ConusEtal11WaveArxiv,ConusKhosh10Farthest}. We will only assume that the initial position $g$ belongs to $L_{loc}^2 \left(\R\right)$, the set of locally square integrable Borel functions, and the initial velocity $\mu$ belongs to $\calM\left(\R\right)$, the set of locally finite Borel measures. These assumptions are natural since the weak solution to the homogeneous wave equation is
\begin{align}\label{E4:J0-Wave}
 J_0(t,x) := \frac{1}{2}\left(g(x+\kappa t)+g(x-\kappa t)\right) +
(\mu*G_\kappa(t,\circ))(x)\;,
\end{align}
where 
$$
   G_\kappa(t,x) = \frac{1}{2} H(t) 1_{[-\kappa t,\kappa t]}(x)
$$ 
is the wave kernel function. Here, $H(t)$ is the Heaviside function (i.e., $H(t)=1$ if $t\ge 0$ and $0$ otherwise), and $*$ denotes convolution in the space variable. 

   Regarding the spde \eqref{E4:Wave}, we interpret it in the integral (mild)
form \eqref{E:Int}:
\begin{equation}\label{E4:WaveInt}
 u(t,x) = J_0(t,x) + I(t,x),
\end{equation}
where
$$
   I(t,x) := \iint_{[0,t]\times\R} G_\kappa\left(t-s,x-y\right) \rho\left( u\left(s,y\right) \right) W\left(\ud s,\ud y\right).
$$

	We show that all the assumptions of Section \ref{SS3:Assumption} are verified
for this equation. More importantly, the abstract bounds take an explicit form
since the function $\calK$ can be evaluated explicitly (see Theorem
\ref{T4:ExUni}). This was also the case for the stochastic heat equation
\cite{ChenDalang13Heat}, but the formula for $\calK$ here is quite different
than in this reference. We also obtain explicit formulas for the second moment
of the solution in the hyperbolic Anderson model, as well as sharp bounds
for higher moments. 
These bounds also apply to other choices of $\rho$. For some particular choices of initial data (such as constant initial position and velocity, or vanishing initial position and Dirac initial velocity), the second moment of the solution takes a particularly simple form (see Corollaries \ref{C4:HomeInit} and \ref{C4:DeltaInit} below).

As an immediate consequence of Theorem \ref{T4:ExUni}, we obtain the result $\overline{m}_p \leq C p^{3/2}$ for $p \geq 2$ of \cite{ConusEtal11WaveArxiv} (see Theorem \ref{T4:Intermit}). We extend their lower bound on the upper Lyapunov exponent $\overline{m}_2$ to the lower Lyapounov exponent, by showing that $\underline{m}_2 \geq c (\kappa/2)^{1/2}$. In the case of the Anderson model $\rho(u) = \lambda u$, we show that $\overline{m}_2 = \underline{m}_2 = \vert \lambda \vert\, (\kappa/2)^{1/2}$.

   Concerning exponential growth indices, we use Theorem \ref{T4:ExUni} to give specific upper and lower bounds on these indices. For instance, we show in Theorem \ref{Tw:Growth} that if the initial position and velocity are bounded below by $c e^{-\beta \vert x \vert}$ and above by $C e^{-\tilde \beta \vert x \vert}$, with $ \beta \geq \tilde \beta$, then
$$
   \kappa \left(1 + \frac{l^2}{8 \kappa \beta^2} \right)^{\frac{1}{2}} \leq \underline{\lambda}(p) \leq \overline{\lambda}(p) \leq \kappa \left(1 + \frac{L^2}{8 \kappa \tilde\beta^2} \right)^{\frac{1}{2}},
$$
for certain explicit constants $l$ and $L$. In the case of the Anderson model $\rho(u) = \lambda u$ and for $p = 2$ and $\beta = \tilde \beta$, we obtain 
$$
   \underline{\lambda}(2) = \overline{\lambda}(2) = \kappa\left (1 + \frac{\lambda^2}{8 \kappa \beta^2} \right)^{1/2}.
$$
Since the exponential growth indices of order $2$ depend on the asymptotic behavior of $E(u(t,x)^2)$ as $t \to \infty$, this equality highlights, in a somewhat surprising way, how the initial data significantly affects the behavior of the solution for all time, despite the presence of the driving noise.

A final question concerns the sample path regularity properties. Denote by $C_{\beta_1,\beta_2}(D)$ the set of  trajectories that are $\beta_1$-H\"older continuous in time and $\beta_2$-H\"older continuous in space on the domain $D\subseteq \R_+\times\R$, and let
\[
   C_{\beta_1-,\beta_2-}(D) := \cap_{\alpha_1\in \;\left]0,\beta_1\right[}
\cap_{\alpha_2\in \;\left]0,\beta_2\right[} C_{\alpha_1,\alpha_2}(D)\;.
\]
Carmona and Nualart \cite[p.484--485]{CarmonaNualart88Smooth} showed that if the initial position is constant and
the initial velocity vanishes, then the solution is in
$C_{1/2-,1/2-}(\R_+\times\R)$ a.s. This property can also be deduced from \cite[Theorem 4.1]{SanSoleSarra99Path}. The case where the spatial domain is $\R^3$ has been studied in \cite{DalangSanzSole09Holder,Dalang09Mini}.

   In \cite{ConusEtal11WaveArxiv}, Conus {\em et al} establish
H\"older-continuity properties of $x \mapsto u(t,x)$ ($t$ fixed). In particular,
they show that if the initial position $g$ is a $1/2$-H\"older-continuous
function
and the initial velocity is square-integrable, then $x \mapsto u(t,x)$ is
$(\frac{1}{2} - \epsilon)$-H\"older-continuous. The assumption on the initial
data is needed, since the H\"older-continuity properties of the initial position
are not smoothed out by the wave kernel but are transferred to $J_0(t,x)$ via
formula \eqref{E4:J0-Wave}.
	
	A related question concerns the stochastic term $I(t,x)$ of
\eqref{E4:WaveInt}, which represents the difference $u(t,x) - J_0(t,x)$ between
the solution of \eqref{E4:Wave} and the solution to the homogeneous wave
equation. We are interested in understanding how properties of the initial data
affect the regularity of $(t,x) \mapsto I(t,x)$. We show in Theorem
\ref{T4:Holder} that the better the (local) integrability properties of the
initial position $g$, the better the regularity of $(t,x) \mapsto I(t,x)$. In
particular, if $g \in L^{2\gamma}_{loc}(\R)$, $\gamma\geq 1$, and $\mu \in
\calM(\R)$, then  $(t,x) \mapsto I(t,x)$ belongs to $C_{\frac{1}{2\gamma'}
-,\frac{1}{2\gamma'}-}\left(\R_+^*\times\R\right)$, where $\frac{1}{\gamma}+
\frac{1}{\gamma'}=1$. We show in Proposition \ref{P4:H-Optimal} that the
H\"older-exponents $\frac{1}{2\gamma'}$ are optimal.

This paper is organized as follows. In Section \ref{SS:MR-Conv}, we study our
abstract integral equation and present the main result in Theorem
\ref{T3:ExUni}. The application to the stochastic heat equation with
distribution-valued initial data is given in Section \ref{SS:ThetaHeat}.
Section \ref{S:MR-Wave} contains the application to the stochastic wave
equation. The main results on existence, uniqueness and formulas and bounds on
moments are stated in Section \ref{SS:Wave-Result} and proved in Section
\ref{ss:wave-lp}. The weak intermittency property is established in Section
\ref{SS:Wave-WInterm}. The bounds on exponential growth indices are given in
Section \ref{SS:Wave-Exp}, and proved in Section \ref{SS:Wave-GrowInd}.
Finally, Section \ref{S:HolerWave} contains our results on H\"older continuity
of the solution of the stochastic wave equation.

\section{Stochastic integral equation of space-time convolution type}
\label{SS:MR-Conv}

We begin by stating the main assumptions which will be needed in our theorem on
existence, uniqueness and moment bounds.

\subsection{Assumptions}\label{SS3:Assumption}
Let $\left\{
W_t(A):A\in\calB_b\left(\R^d\right),\, t\ge 0 \right\}$
be a space-time white noise
defined on a complete probability space $(\Omega,\calF,P)$, where
$\calB_b\left(\R^d\right)$ is the
collection of Borel sets with finite Lebesgue measure.
Let  $(\calF_t,\, t\ge 0)$ be the standard filtration generated by this space-time
white noise, i.e., $\calF_t = \sigma\left(W_s(A):0\le s\le
t,A\in\calB_b\left(\R^d\right)\right)\vee
\calN$, where $\calN$ is the $\sigma$-field generated by all $P$-null sets in
$\calF$.
We use $\Norm{\cdot}_p$ to denote the
$L^p(\Omega)$-norm.
A random field $Y(t,x)$, $(t,x)\in \R_+^*\times
\R^d$, is said to be {\it adapted} if for all $(t,x)\in\R_+^*\times\R^d$,
$Y(t,x)$ is $\calF_t$-measurable, and it is said to be {\it jointly measurable}
if it is measurable with respect to $\calB(\R_+^*\times\R^d)\times\calF$.
For $p\ge 2$, if $\lim_{\left(t',x'\right)\rightarrow
(t,x)}\Norm{Y(t,x)-Y\left(t',x'\right)}_p
=0$ for all $(t,x) \in \R_+^*\times\R^d$, then
$Y$ is said to be {\it $L^p(\Omega)$-continuous}.

Let  $G$, $J_0: \R_+\times\R^d\mapsto \R$ be deterministic Borel functions.
We use the convention that $G(t,\cdot)\equiv 0$ if $t\le 0$.  In the following,
we will use $\cdot$ and $\circ$ to denote the time and space dummy variables
respectively.

\begin{definition}\label{D3:Solution}
A random field $(u(t,x),\, (t,x)\in\R_+\times\R^d)$, is called a
{\it solution} to \eqref{E:Int} if
\begin{enumerate}[(1)]
\item $u(t,x)$ is adapted and jointly measurable;
\item For all $(t,x)\in\R_+^*\times\R^d$,
$\left(G^2(\cdot,\circ) \star \left[\Norm{\rho(u(\cdot,\circ))}_2^2
\theta^2(\cdot,\circ)\right] \right)(t,x)<+\infty$, where  $\star$ denotes
the simultaneous convolution in both space and time variables,
and the function $(t,x)\mapsto I(t,x)$ from $\R_+\times\R^d$ into
$L^2(\Omega)$ is continuous;
\item $u(t,x)=J_0(t,x)+I(t,x)$, where for all $(t,x)\in\R_+\times\R^d$,
\begin{align}\label{E:WalshSI}
 I(t,x) =  \iint_{\R_+\times\R^d} G\left(t-s,x-y\right)
\rho\left( u\left(s,y\right)\right)
\theta\left(s,y\right)W\left(\ud s,\ud y\right),\quad\text{a.s.}
\end{align}
\end{enumerate}
We call $I(t,x)$ the {\it stochastic integral part} of the random field
solution. This stochastic integral is interpreted in the sense of Walsh
\cite{Walsh86}.
\end{definition}

\begin{remark}
Consider the stochastic wave equation \eqref{E4:Wave}
with $g\in L^2_{loc}\left(\R\right)$ and $\mu=0$.
In this case, $J_0(t,x) = 1/2
\left(g(\kappa t +x)+g(\kappa t -x)\right)$. Since the
initial position $g$ may not be defined for every $x$, the function $(t,x) \mapsto J_0(t,x)$ may not be defined for certain $(t,x)$.
Therefore, for these $(t,x)$, $u(t,x)$ may not be well-defined (see Example
\ref{Ex4:wave-g-1/4}). Nevertheless, as we will show later, $I(t,x)$ is always
well
defined for each $(t,x)\in\R_+\times\R$, and in most cases (when Assumption
\ref{A:Holder} below holds), it has a continuous version.
Finally, we remark that for the stochastic heat equation with deterministic
initial conditions, this problem does not arise because in that equation,
$(t,x)\mapsto J_0(t,x)$ is continuous over $\R_+^*\times\R$ thanks
to the smoothing effect of the heat kernel.
\end{remark}

As in \cite{Dalang99Extending}, a very first issue is whether the linear
equation, where $\rho(u)\equiv 1$, admits a random field solution.
For $t\in \R_+$, and $x,y\in\R^d$, this leads to examining the quantity
\begin{align}\label{E:Theta}
\Theta(t,x,y):= \iint_{[0,t]\times\R^{d}}\ud s \ud z \:
G(t-s,x-z) G\left(t-s,y-z\right) \theta^2(s,z) \;.
\end{align}
Clearly, $2 \Theta(t,x,y) \le \Theta(t,x,x)+\Theta\left(t,y,y\right)$.

\begin{assumption}
\label{A:R-G}
$G(t,x)$ is such that\\
(i)
$\Theta(t,x,x) <+\infty$ for all $(t,x)\in\R_+\times\R^d$;\\
(ii)
$\lim_{\left(t',x'\right)\rightarrow (t,x)}
G\left(t',x'\right) = G(t,x)$,
for almost
all $(t,x)\in\R_+\times\R^d$.
\end{assumption}

If $\theta(t,x)\equiv 1$, $d=1$ and if the
underlying partial differential operator is $\frac{\partial}{\partial t}-\calA$,
where $\calA$ is the generator of a real-valued L\'evy
process with the L\'evy exponent $\Psi(\xi)$, then Assumption \ref{A:R-G} (i) is
equivalent to $\frac{1}{2\pi}\int_\R \frac{\ud \xi}{\beta+2\Re \Psi(\xi)}
<+\infty$, for all $\beta>0$,
where $\Re \Psi(\xi)$ is the real part of $\Psi(\xi)$:
see \cite{Dalang99Extending,FoondunKhoshnevisan08Intermittence}.
For the one-dimensional stochastic heat equation studied in
\cite{ChenDalang13Heat}, it is also clearly satisfied.
For the stochastic wave equation \eqref{E4:Wave}, this assumption also holds: see
\eqref{E4:Theta}.

\begin{assumption}\label{A:InDt}
For all
compact sets $K\subseteq \R_+^*\times\R^d$ and all integers $p\ge 2$,
\[
\sup_{(t,x)\in K}\left(\left(\left[1+J_0^2\right]\theta^2 \right)\star
G^2 \right)(t,x)<+\infty.
\]
\end{assumption}

   We note that a related assumption appears in \cite[Proposition 1.8]{CarmonaNualart88Smooth}.
The next three assumptions will be used to establish the
$L^p(\Omega)$-continuity in a Picard iteration.
Assumption \ref{A:Cont-Bdd} is for kernel functions
similar to the wave kernel and Assumptions \ref{A:Cont-Tail}--\ref{A:InDt-Bdd}
are for those similar to the heat kernel.
We need some notation:
for $\beta\in \;]0,1[\:$, $\tau>0$, $\alpha>0$ and $(t,x)\in\R_+^*\times\R^d$,
define
\begin{align}\label{E:B}
B_{t,x,\beta,\tau,\alpha}:=\left\{\left(t',x'\right)\in\R_+^*\times\R^d:\: \beta
t \le  t' \le t+\tau,\: \left|x-x'\right|\le \alpha\right\}\:.
\end{align}

\begin{assumption}[Uniformly bounded kernel functions]\label{A:Cont-Bdd}
There exist three constants $\beta\in\; ]0,1[\:$, $\tau>0$ and
$\alpha>0$ such that for all $(t,x)\in\R_+^*\times\R^d$, for some constant
$C>0$, we have for all
$\left(t',x'\right)\in B_{t,x,\beta,\tau,\alpha}$ and all $\left(s,y\right)\in
[0,t'[\times\R^d$, $G(t'-s,x'-y) \le C \: G(t+1-s,x-y)$.
\end{assumption}

\begin{assumption}[Tail control of kernel functions]\label{A:Cont-Tail}
There exists $\beta\in \;]0,1[$ such that for all
$(t,x)\in\R_+^*\times\R^d$, for some constant $a>0$, we have for all
$\left(t',x'\right)\in B_{t,x,\beta,1/2,1}$ and all $s \in [0,t'[$ and $y\in
\R^d$ with
$|y|\ge a$, $G(t'-s,x'-y) \le G(t+1-s,x-y)$.
\end{assumption}

\begin{assumption}\label{A:Continuous}
For all $(t,x)\in\R_+^*\times\R^d$, 
\begin{gather*}
\lim_{\left(t',x'\right)\rightarrow (t,x)}
\iint_{\R_+\times\R^{d}}\ud s\ud y \:
\theta(s,y)^2\left(
G(t'-s,x'-y)-G\left(t-s,x-y\right)
\right)^2 =0 .
\end{gather*}
\end{assumption}

Note that this assumption can be more explicitly expressed in the following way:
\begin{multline}\label{E:A-Cont}
\int_0^{t_*} \ud s\int_{\R^{d}}\ud y \:
\theta(s,y)^2 \left(
G(t'-s,x'-y)-G\left(t-s,x-y\right)
\right)^2\\
+\int_{t_*}^{\hat{t}} \ud s\int_{\R^{d}}\ud y\:
\theta(s,y)^2
G^2\left(\hat{t}-s,\hat{x}-y\right) \rightarrow 0,
\end{multline}
as $(t',x')\rightarrow (t,x)$, where
\begin{align}\label{E:txtx}
  \left(\: t_*,x_*\: \right) =
  \begin{cases}
    \left(t',x'\right) &    \text{if $t'\le t$,}\cr
    (t,x)  &    \text{if $t'> t$,}
  \end{cases}
  \quad\text{and}\quad
  \left(\hat{t},\hat{x}\right) =
  \begin{cases}
    (t,x) &    \text{if $t'\le t$.}\cr
    \left(t',x'\right)  &    \text{if $t'> t$.}
  \end{cases}
\end{align}


\begin{assumption}\label{A:InDt-Bdd}
For all compact sets $K \subseteq
\R_+^*\!\times\R^d$, $\sup_{(t,x)\in K} \left|J_0(t,x)\right| <\infty$.
\end{assumption}

   The remaining assumptions are mainly needed for control of the moments of the solution. We introduce some notation. For two functions $f, g : \R_+\times\R^d\mapsto \R_+$, define their
{\it $\theta$-weighted space-time convolution} by
\[\left(f\rhd g\right)(t,x) :=
\left(\left(\theta^2 f \right)\star
g\right)(t,x),\quad\text{for all $(t,x)\in\R_+\times\R^d$},
\]
In the following, $f(t,x)$ will play the role of $J_0^2(t,x)$, and $g(t,x)$ of
$G^2(t,x)$. In the Picard iteration scheme, the expression
$\left(
\left(
\cdots
\left(
\left(f\rhd g_1 \right)
\rhd g_2 \right)\rhd
\cdots
\right)\rhd g_n
\right) (t,x)$ will appear, where $g_i = g$.
Since $\rhd$ is {\em not} associative
in general (contrary to the case $\theta \equiv 1$), we need to handle this formula with care.

\begin{definition}
Let $n\ge 2$ and let $g_k:\R_+\times\R^d
\mapsto \R_+$, $k=1,\dots, n$.
Define the
{\em $\theta$-weighted multiple space-time convolution},
for $(t,x)$, $\left(s,y\right)\in\R_+\times\R^d$ with $0\le s \le t$,
by
\begin{align} \nonumber
&\rhd_n\left(g_1, g_2 , \dots, g_n\right)
\left(t,x;s,y\right)\\ \nonumber
&\quad :=
\int_0^s \ud s_{n-1}\int_{\R^d}\ud y_{n-1}\:
g_n\left(s-s_{n-1},y-y_{n-1}\right)
\theta^2\left(t-s+s_{n-1},x-y+y_{n-1}\right)
\\ \nonumber
&\quad\qquad \times\int_0^{s_{n-1}}\ud s_{n-2}\int_{\R^d}\ud y_{n-2}\:
g_{n-1}\left(s_{n-1}-s_{n-2},y_{n-1}-y_{n-2}\right)
\theta^2\left(t-s+s_{n-2},x-y+y_{n-2}\right)
\\ \nonumber
&\qquad\qquad \times \cdots\cdots\times\int_0^{s_3}\ud s_2\int_{\R^d} \ud y_2\:
g_3\left(s_3-s_2,y_3-y_2\right)
\theta^2\left(t-s+s_2,x-y+y_2\right)
\\ 
&\qquad\qquad \times\int_0^{s_2}\ud s_1\int_{\R^d}\ud y_1\:
g_2\left(s_2-s_1,y_2-y_1\right)
\theta^2\left(t-s+s_1,x-y+y_1\right)
g_1 \left(s_1,y_1\right).
\label{E:FORRHD}
\end{align}
\end{definition}
\vspace{1em}

Notice that
\[\rhd_n\left(g_1,\dots,g_n\right)(t,x;t,x)
=\left(
\left(
\cdots
\left(
\left(g_1\rhd g_2\right)\rhd g_3
\right)\rhd \cdots
\right)
\rhd g_n
\right)(t,x),\]
where the r.h.s. has $n-1$ convolutions.
By the change of variables
\begin{equation}
 \label{E:ChangeVar}
\begin{aligned}
& \tau_1 = s-s_{n-1},& \tau_2 &= s-s_{n-2},& &\cdots, & \tau_{n-1} &=
s-s_{1}\;,\quad \text{and}\\
& z_1 = y-y_{n-1},& z_2 &= y-y_{n-2},& &\cdots, & z_{n-1} &=
y-y_{1}\;,
\end{aligned}
\end{equation}
and Fubini's theorem,
the multiple convolution $\rhd_n$ has an equivalent definition:
\begin{align}
\begin{aligned}
\rhd_n&\left(g_1, g_2 , \dots, g_n\right)
\left(t,x;s,y\right)\\
&=
\int_0^s \ud \tau_{n-1}\int_{\R^d}\ud z_{n-1}\:
\theta^2\left(t-\tau_{n-1},x-z_{n-1}\right)
g_1\left(s-\tau_{n-1},y-z_{n-1}\right)
\\
&\quad\times\int_0^{\tau_{n-1}}\ud \tau_{n-2}\int_{\R^d}\ud z_{n-2}\:
\theta^2\left(t-\tau_{n-2},x-z_{n-2}\right)
g_2\left(\tau_{n-1}-\tau_{n-2},z_{n-1}-z_{n-2}\right)
\\
&\quad\times \cdots\cdots
\times\int_0^{\tau_3}\ud \tau_2 \int_{\R^d}\ud z_2\:
\theta^2\left(t-\tau_2,x-z_2\right)
g_{n-2}\left(\tau_3-\tau_2,z_3-z_2\right)
\\
&\quad\times\int_0^{\tau_2}\ud \tau_1\int_{\R^d} \ud z_1\:
\theta^2\left(t-\tau_1,x-z_1\right)
g_{n-1}\left(\tau_2-\tau_1,z_2-z_1\right)
g_n\left(\tau_1,z_1\right)\;.
\end{aligned}
\label{E:BACKRHD}
\end{align}

\begin{lemma}
\label{L3:RhdAs}
Let $f, g_k:\R_+\times\R^d\mapsto\R_+$, $k=1,\dots,n+1$, and $n\ge 2$.
Then for all $(t,x)\in\R_+\times\R^d$, we have
\begin{equation}
\left(
\left(
\cdots
\left(
\left(f\rhd g_1 \right)
\rhd g_2 \right)\rhd \cdots
\right)\rhd g_n
\right) (t,x)
\label{E:RhdAs1}
= \left(f\rhd
\rhd_n\left(g_1,\dots,g_n\right)(t,x;\cdot,\circ)\right)(t,x),
\end{equation}
\begin{equation}
\left(f\rhd
\rhd_n\left(g_1,\dots,g_n\right)(t,x;\cdot,\circ)\right)(t,x)
=
\left(\left(f\rhd g_1\right)\rhd
\rhd_{n-1}\left(g_2,\dots,g_n\right)(t,x;\cdot,\circ)\right)(t,x),
\label{E:RhdAs10}
\end{equation}
and
\begin{multline}\label{E:RhdAs2}
\int_0^t \ud s\int_{\R^d}\ud y
\left(f\rhd
\rhd_n\left(g_1,\dots,g_n\right)(s,y;\cdot,\circ)\right)(s,y)
\;\theta^2(s,y)g_{n+1}\left(t-s,x-y\right)
\\
=
\left(f\rhd
\rhd_{n+1}\left(g_1,\dots,g_{n+1}\right)(t,x;\cdot,\circ)\right)(t,x).
\end{multline}
\end{lemma}
Note that $(s,y)$ appears twice in the term $f\rhd\rhd_n(\cdots)$ on the
l.h.s. of \eqref{E:RhdAs2}.
The proof of Lemma \ref{L3:RhdAs} is straightforward; see
\cite[Lemma 3.2.6]{LeChen13Thesis} for details.
When $n=2$, for $f$ and $g : \R_+\times\R^d\mapsto \R_+$, $\rhd_2(f,g)(t,x;t,x)
= \left(f\rhd g\right)(t,x)$ and
\begin{align}
\rhd_2\left(f,g\right) (t,x; s,y)
&=\int_0^s\ud s_0\int_{\R^{d}} \ud y_0\:
g\left(s-s_0,y-y_0\right)
\theta^2 \left(t-s+s_0,x-y+y_0\right)
f\left(s_0,y_0\!\right)
\label{E:ForRHD}
\\
\label{E:BackRHD}
&=\int_0^s\ud \tau_0  \int_{\R^{d}}\ud z_0\:
\theta^2 \left(t-\tau_0,x-z_0\right)
f\left(s-\tau_0,y-z_0\right)
g\left(\tau_0,z_0\right).
\end{align}
In particular, if $\theta(t,x)\equiv 1$,
then $\rhd_2$ reduces to the standard
space-time convolution $\star$ (as is the case for $\rhd$), in which case the
first two variables $(t,x)$
do not play a role.
We call \eqref{E:ForRHD} and \eqref{E:FORRHD} the {\it forward}
formulas,
and \eqref{E:BackRHD} and \eqref{E:BACKRHD} the {\it backward}
formulas.

For $\lambda\in\R$, define $\calL_0\left(t,x;\lambda\right):=\lambda^2
G^2(t,x)$, and for $n\in\bbN^*$, 
\[\calL_n\left(t,x;s,y;\lambda\right):=
\rhd_{n+1}\big(\calL_0(\cdot,\circ;\lambda),
\dots,\calL_0(\cdot,\circ;\lambda)\big)\left(t,x;s,y\right)\]
for all $(t,x),\left(s,y\right)\in\R_+^*\times \R^d$ with $s\le t$.
By convention, $\calL_0\left(t,x;s,y;\lambda\right) = \lambda^2
G^2\left(s,y\right)$.
For $n\in\bbN$, define
\[\calH_n\left(t,x;\lambda\right) := \left(1\rhd \calL_n(t,x;
\cdot,\circ;\lambda)\right)\left(t,x\right).
\]
By definition, both
$\calL_n$ and $\calH_n$ are non-negative.
We use the following conventions:
\begin{equation}\label{E:Convention}
\begin{aligned}
 \calL_n\left(t,x;s,y\right) &:= \calL_n\left(t,x;s,y;\: \lambda\right),&
 \overline{\calL}_n\left(t,x;s,y\right)&:=
\calK\left(t,x;s,y;\: \Lip_\rho \right),\\
\underline{\calL}_n\left(t,x;s,y\right)&:=
\calL_n\left(t,x;s,y;\: \lip_\rho \right), &
\widehat{\calL}_n\left(t,x;s,y\right)&:=
\calL_n\left(t,x;s,y;\: a_{p,\Vip}\: z_p\: \Lip_\rho \right),\;
\text{$p\ge 2$}\;,
\end{aligned}
\end{equation}
where the constant $a_{p,\Vip} (\le 2)$ is defined by
\begin{align}\label{E:a_pv}
a_{p,\Vip} \: :=\:
\begin{cases}
2^{(p-1)/p}& \Vip\ne 0,\: p>2,\cr
\sqrt{2} & \Vip =0,\: p>2,\cr
1 & p=2,
\end{cases}
\end{align}
and $z_p$ is the optimal universal constant in the
Burkholder-Davis-Gundy inequality (see \cite[Theorem 1.4]{ConusKhosh10Farthest})
and so $z_2=1$ and $z_p\le 2\sqrt{p}$ for all $p\ge 2$.
Note that the kernel function $\widehat{\calL}_n\left(t,x;s,y\right)$ depends on
the parameters $p$ and $\Vip$, which is usually clear from the context.
Similarly, define $\overline{\calH}_n(t,x)$, $\underline{\calH}_n(t,x)$ and
$\widehat{\calH}_n(t,x)$.
The same conventions will apply to
$\calK\left(t,x;s,y\right)$,
$\overline{\calK}\left(t,x;s,y\right)$, $\underline{\calK}\left(t,x;s,y\right)$
and $\widehat{\calK}\left(t,x;s,y\right)$
below.

\begin{assumption}\label{A:Converge}
The kernel functions $\calL_n\left(t,x;s,y;\lambda\right)$ and
$\calH_n(t,x;s;\lambda)$,
with $n\in\bbN$ and $\lambda\in\R$, are well defined and
the sum of $\calL_n\left(t,x;s,y;\lambda\right)$ converges
for all
$(t,x)$ and $(s,y)\in\R_+^*\times\R^d$ with $s\le t$. Denote this sum
by
\[\calK\left(t,x;s,y;\lambda\right):= \sum_{n=0}^\infty
\calL_n\left(t,x;s,y;\lambda\right).\]
\end{assumption}

The next assumption is a convenient assumption which will guarantee the
continuity of the function $(t,x)\mapsto I(t,x)$ from $\R_+\times\R^d$ into
$L^p(\Omega)$ for $p\ge 2$.

\begin{assumption}\label{A:Gronwall}
There are non-negative functions $B_n(t):=B_n(t;\lambda)$ such that
(i) $B_n(t)$ is nondecreasing in $t$;
(ii) for all
$(t,x)$, $(s,y)\in\R_+^*\times\R^d$ with $s\le t$ and $n\in\bbN$,
$\calL_n\left(t,x;s,y\right) \le \calL_0\left(s,y\right) B_n(t)$
(set $B_0(t)\equiv 1$);
(iii) $\sum_{n=0}^\infty \sqrt{B_n(t)}<+\infty$, for all $t>0$.
\end{assumption}

The above assumption guarantees that the following function (without any
square root) is well defined:
\begin{align}
 \label{E:Upsilon}
\Upsilon\left(t;\lambda\right) := \sum_{n=0}^\infty
B_n\left(t;\lambda\right)\;,\quad t\ge 0.
\end{align}
We use the same conventions on the parameter $\lambda$
for the function $\Upsilon(t;\lambda)$.
Clearly, for all $(t,x)$ and $\left(s,y\right)\in\R_+\times\R^d$ such that $s\le
t$,
\begin{align}\label{E:KUP}
\calK\left(t,x;s,y\right) \le \Upsilon(t) \calL_0\left(s,y\right).
\end{align}
Another consequence of Assumption \ref{A:Gronwall} is that
$\sum_{n=0}^\infty \calH_n(t,x)\le \calH_0(t,x) \Upsilon(t)
<+\infty$ for all $(t,x)\in\R_+\times\R^d$ and $0\le s\le t$,
and so the function $\calH(t,x):= \left(1\rhd
\calK(t,x;\cdot,\circ)\right)\left(t,x\right)$ is well defined
and equals $\sum_{n=0}^\infty \calH_n(t,x)$ by the
monotone convergence theorem.

The following chain of inequalities is a direct consequence of Assumption
\ref{A:R-G} and the observations
above: for all $n\in\bbN$, and all
$(t,x)$, $(s,y)\in\R_+^*\times\R^d$ with $s\le t$,
\begin{gather}
 \left(J_0^2\rhd \calL_n(t,x;\cdot,\circ)\right) \left(t,x\right)
\le
\left( J_0^2\rhd \calK(t,x;\cdot,\circ) \right) \left(t,x\right)
\le
\Upsilon(t) \left( J_0^2\rhd \calL_0\right) \left(t,x\right)
<+\infty\:.
 \label{E:IKn-fin}
\end{gather}

\subsection{Main theorem}\label{SS3:MainRes}
Assume that $\rho:\R\mapsto \R$ is globally Lipschitz
continuous with Lipschitz constant $\LIP_\rho>0$.
We need some growth conditions on $\rho$:
Assume that for some constants $\Lip_\rho>0$ and $\Vip \ge 0$,
\begin{align}\label{E:LinGrow}
|\rho(x)|^2 \le \Lip_\rho^2 \left(\Vip^2 +x^2\right),\qquad \text{for
all $x\in\R$}\;,
\end{align}
Note that $\Lip_\rho\le \sqrt{2}\LIP_\rho$, and the inequality may be strict.
In order to bound the second moment from below,
we will sometimes assume that for some constants $\lip_\rho>0$ and $\vip \ge 0$,
\begin{align}\label{E:lingrow}
|\rho(x)|^2\ge \lip_{\rho}^2\left(\vip^2+x^2\right),\qquad \text{for
all $x\in\R$}\;.
\end{align}
We shall also give particular attention to the Anderson model, which is a
special case of
the
following quasi-linear growth condition: for some constants $\vv\ge 0$ and
$\lambda\ne 0$,
\begin{align}\label{E:qlinear}
|\rho(x)|^2= \lambda^2\left(\vv^2+x^2\right),\qquad  \text{for
all $x\in\R$}\;.
\end{align}

To facilitate stating the theorem, we group the assumptions above as follows:
\begin{enumerate}
 \item[\conRef{G}] (General conditions):
\begin{enumerate}
 \item $G(t,x)$ satisfies Assumptions \ref{A:R-G}, \ref{A:Converge},
and
\ref{A:Gronwall};
 \item $J_0(t,x)$ and $\theta(t,x)$ satisfy Assumption \ref{A:InDt}.
\end{enumerate}
\item[\conRef{W}] (Wave type)
$G(t,x)$ satisfies Assumptions \ref{A:Cont-Bdd}.
\item [\conRef{H}] (Heat type):
\begin{enumerate}[(a)]
\item $G(t,x)$ satisfies Assumptions \ref{A:Cont-Tail} and
\ref{A:Continuous};
\item $J_0(t,x)$ satisfies Assumption \ref{A:InDt-Bdd}.
\end{enumerate}
\end{enumerate}

\begin{theorem}\label{T3:ExUni}
Suppose the function $\rho(u)$ is Lipschitz continuous and satisfies the growth
condition
\eqref{E:LinGrow}.
If \conRef{G} and at least
one of \conRef{W} and \conRef{H} hold,
then the stochastic integral equation \eqref{E:Int} has a
solution
\[
\left\{\;u(t,x)=J_0(t,x)+I(t,x):\; t>0,\; \;x \in\R^d\:\:\right\}
\]
in the sense of Definition \ref{D3:Solution}.
This solution has the following properties:\\
(1) $I(t,x)$ is unique (in the sense of versions).\\
(2) $I(t,x)$ is $L^p(\Omega)$--continuous over $\R_+\times\R^d$ for all
integers $p\ge 2$.\\
(3) For all even integers $p\ge 2$, $t>0$, and $x,y\in\R^d$,
\begin{gather}
\label{E:Mom-Up}
\Norm{u(t,x)}_p^2 \le
\begin{cases}
J_0^2(t,x) + \left(J_0^2\rhd
\overline{\calK}(t,x;\cdot,\circ) \right) (t,x) +
\Vip^2 \: \overline{\calH}(t,x),& \text{if $p=2$,}\\[0.6em]
2 J_0^2(t,x) + \left(2 J_0^2 \rhd
\widehat{\calK}(t,x;\cdot,\circ) \right) (t,x) +
\Vip^2 \: \widehat{\calH}(t,x),& \text{if $p>2$,}
\end{cases}
\end{gather}
and
\begin{multline}
\label{E:TP-Up}
\E\left[u(t,x)u(t,y)\right]  \le
J_0(t,x)J_0(t,y) + \Lip_\rho^2\: \Vip^2\: \Theta(t,x,y)\\
+\Lip_\rho^2\int_0^t\ud s \int_{\R^d}\ud z\:
\overline{f}(s,z) \: \theta^2(s,z) \: G(t-s,x-z) G\left(t-s,y-z\right) ,
\end{multline}
where $\overline{f}(s,z)$ denotes the r.h.s. of \eqref{E:Mom-Up} for $p=2$.\\
(4) If $\rho$ satisfies \eqref{E:lingrow}, then for all $t>0$, and
$x,y\in\R^d$,
\begin{align}
\label{E:SecMom-Lower}
\Norm{u(t,x)}_2^2 \ge J_0^2(t,x) + \left(J_0^2 \rhd
\underline{\calK}(t,x;\cdot,\circ) \right)
(t,x)+
\vip^2\: \underline{\calH}(t,x),
\end{align}
and
\begin{multline}
\label{E:TP-Lower}
\E\left[u(t,x)u(t,y)\right] \ge
J_0(t,x)J_0(t,y) + \lip_\rho^2\: \vip^2\: \Theta(t,x,y)\\
+\lip_\rho^2\int_0^t\ud s \int_{\R^d}\ud z\: \underline{f}(s,z) \:
\theta^2(s,z)\:
G(t-s,x-z) G\left(t-s,y-z\right) ,
\end{multline}where $\underline{f}(s,z)$ denotes the r.h.s. of
\eqref{E:SecMom-Lower}.\\
(5) In particular, for the quasi-linear case $|\rho(u)|^2=\lambda^2
\left(\vv^2+u^2\right)$, for all $t>0$, $x$, $y\in\R^d$,
\begin{align}
 \label{E:SecMom}
 \Norm{u(t,x)}_2^2 = J_0^2(t,x)
 + \left(J_0^2 \rhd \calK(t,x;\cdot,\circ) \right)
(t,x)+ \vv^2 \:\calH(t,x),
\end{align}
and
\begin{multline}
\label{E:TP}
\E\left[u(t,x)u(t,y)\right] =
J_0(t,x)J_0(t,y) + \lambda^2 \vv^2 \: \Theta(t,x,y) \\
+ \lambda^2 \int_0^t\ud s \int_{\R^d} \ud z\: f(s,z) \: \theta^2(s,z)\:
G(t-s,x-z) G\left(t-s,y-z\right),
\end{multline}
where $f(s,z)=\Norm{u(s,z)}_2^2$ is given in
\eqref{E:SecMom}.
\end{theorem}

We now present an assumption that will imply H\"older continuity of the
stochastic integral part of the solution $u$ of \eqref{E:Int}.

\begin{assumption}(Sufficient conditions for H\"older continuity)
\label{A:Holder}
Given $J_0(t,x)$ and $v\in\R$, assume that there are $d+1$ constants $\gamma_i
\in\; ]0,1]$, $i=0,\dots,d$ such that for all $n>1$,
one can find a finite constant $C_{n}<+\infty$
such that for all $(t,x)$ and
$\left(t',x'\right)\in K_n:= [1/n,n]\times [-n,n]^d$ with $t<t'$, we
have that
\begin{multline}
\label{E:H-135}
 \iint_{\R_+\times\R^d}\ud s\ud y\:
\left(v^2+2J_0^2\left(s,y\right)\right)
\left(G\left(t-s,x-y\right)-G(t'-s,x'-y)
\right)^2 \theta^2\left(s,y\right)\\
\le
C_n \: \tau_{\gamma_0,\dots,\gamma_d}\left((t,x),\left(t',x'\right)\right),
\end{multline}
and
\begin{multline}\label{E:H-246}
 \iint_{\R_+\times\R^d}\ud s\ud y\:
\left(\left(v^2+2J_0^2\right)\rhd G^2 \right)\left(s,y\right)
\left(G\left(t-s,x-y\right)-G(t'-s,x'-y)
\right)^2 \theta^2\left(s,y\right)\\
\le
C_n \: \tau_{\gamma_0,\dots,\gamma_d}\left((t,x),\left(t',x'\right)\right)\;,
\end{multline}
where $\tau_{\gamma_0,\dots,\gamma_d}\left((t,x),\left(t',x'\right)\right):=
\left|t-t'\right|^{\gamma_0} + \sum_{i=1}^d \left|x_i-x_i'\right|^{\gamma_i}$.
\end{assumption}

The following lemma is useful for verifying Assumption \ref{A:Holder}.
Its proof is straightforward and we leave it to the interested reader.
\begin{lemma}\label{L3:Holder-Equiv}
Assumption \ref{A:Holder} is equivalent to the following statement:
Given $J_0$ and $v\in\R$, assume that there are $d+1$ constants
$\gamma_i \in \: ]0,1]$, $i=0,\dots,d$ such that for all $n>1$,
one can find six finite constants $C_{n,i}<+\infty$, $i=1,\dots,6$,
such that for all $(t,x)$ and $\left(t+h,x+z\right)\in
K_n:=[1/n,n]\times [-n,n]^d$
with $h> 0$, we have,
\begin{gather}\label{E:H-1}
 \left(\left(v^2+2J_0^2\right)\rhd \left(G(\cdot,\circ) -
G(\cdot+h,\circ)\right)^2\right)(t,x)  \le C_{n,1} \: h^{\gamma_0}\;,\\
\label{E:H-3}
\left(\left(v^2+2J_0^2\right)\rhd \left(G(\cdot,\circ) -
G(\cdot,\circ+z)\right)^2\right)(t,x)  \le C_{n,3} \sum_{i=1}^d
|z_i|^{\gamma_i}\;,\\
\label{E:H-5}
 \iint_{[t,t+h]\times\R^d}\ud u\ud y \:
\left(v^2+2J_0^2(u,y)\right)\! G^2(t+h-u,x+z-y) \theta^2(u,y) \le
C_{n,5} \: h^{\gamma_0},\\
\notag
\left(\left[\left(v^2+2J_0^2\right)\rhd G^2\right] \rhd \left(G(\cdot,\circ) -
G(\cdot+h,\circ)\right)^2\right)(t,x)  \le C_{n,2}\: h^{\gamma_0}\;,\\ \notag
\left(\left[\left(v^2+2J_0^2\right)\rhd G^2 \right]\rhd \left(G(\cdot,\circ) -
G(\cdot,\circ+z)\right)^2\right)(t,x) \le C_{n,4} \sum_{i=1}^d
|z_i|^{\gamma_i}\;,\\ \notag
 \iint_{[t,t+h]\times\R^d}\ud u\ud y
\left(\left(v^2+2J_0^2\right) \rhd G^2 \right)(u,y)\; G^2(t+h-u,x+z-y)
\theta^2(u,y)   \le
C_{n,6} \: h^{\gamma_0}.
\end{gather}
\end{lemma}

\begin{theorem}\label{T3:Holder}
Suppose that the conditions of Theorem \ref{T3:ExUni} hold. If, in addition,
Assumption \ref{A:Holder} is also satisfied, then for all compact sets
$K\subseteq \R_+^*\times\R^d$ and all $p\ge
1$, there is a constant $C_{K,p}$ such that for all $(t,x)$, $(t',x')\in K$,
\[
\Norm{I(t,x)-I(t',x')}_p\le C_{K,p}
\left[\tau_{\gamma_0,\dots,\gamma_d}\left((t,x),(s,y)\right)\right]^{1/2},
\]
and therefore $(t,x)\mapsto I(t,x)$ belongs to 
$C_{
\frac{\gamma_0}{2}-,\frac{\gamma_1}{2}-,\dots,\frac{\gamma_d}{2}-}
\left(\R_+^*\times\R^d\right)$ a.s. In addition, for $0\le \alpha <
1/2-(1/p)\sum_{i=0}^d \gamma_i^{-1}$,
\[
\E\left[\left(
\mathop{\sup_{(t,x),\: (s,y)\in K}}_{(t,x)\ne (s,y)}
\frac{|I(t,x)-I(s,y)|}{\left[\tau_{\gamma_0,\dots,\gamma_d}((t,x),(s,y))\right]
^\alpha } \right)^p\;\right ]<+\infty.
\]
Moreover, if the compact sets $K_n$ in Assumption \ref{A:Holder} can be chosen
as $[0,n]\times [-n,n]^d$, then
$I(t,x)\in
C_{
\frac{\gamma_0}{2}-,\frac{\gamma_1}{2}-,\dots,\frac{\gamma_d}{2}-}
\left(\R_+\times\R^d\right)$ a.s.
\end{theorem}
\begin{proof}
With Propositions 4.4 and 4.5 of \cite{ChenDalang13Holder} replaced by
Assumption \ref{A:Holder} (or equivalently Lemma \ref{L3:Holder-Equiv}),
the proof is identical to part (1) of Theorem 3.1 in \cite{ChenDalang13Holder}.
For the range of the parameter $\alpha$, see \cite[Theorem 1.4.1]{Kunita90Flow}.
\end{proof}

\subsubsection{Some lemmas and propositions}
Following \cite{Walsh86}, a random field $\{Z(t,x)\}$ is called {\em elementary}
if we can write
$Z(t,x)=Y 1_{]a,b]}(t) 1_{A}(x)$, where $0\le a<b$, $A\subset \R^d$ is a
rectangle, and $Y$ is an $\calF_a$--measurable random variable. A {\it simple}
process is a finite sum of elementary random fields. The set of simple processes
generates the {\em predictable} $\sigma$-field on $\R_+\times\R^d\times\Omega$,
denoted by $\calP$. For $p\ge 2$ and
$X\in L^2\left(\R_+\times\R^d, L^p(\Omega)\right)$, set
\begin{align}\label{E:Norm_Mp}
 \Norm{X}_{M, p}^2 & :
= \iint_{\R_+^*\times\R^d}\ud s\ud y \Norm{X\left(s,y\right)}_p^2<+\infty\;.
\end{align}
When $p=2$, we write $\Norm{X}_M$ instead of $\Norm{X}_{M,2}$.
As pointed out in \cite{ChenDalang13Heat}, $\iint X\ud W$ is defined in
\cite{Walsh86} for predictable $X$
such that $\Norm{X}_M<+\infty$. However, the condition of predictability is not
always so easy to check, and as in the case of ordinary Brownian motion
\cite[Chapter 3]{ChungWilliams90}, it is convenient to be able to integrate
elements $X$ that are jointly measurable and adapted. For this, let $\calP_p$
denote the closure in $L^2\left(\R_+\times\R^d, L^p(\Omega)\right)$ of simple
processes.
Clearly, $\calP_2\supseteq \calP_p \supseteq \calP_q$ for $2\le p\le q<+\infty$,
and according to \Itos isometry, $\iint X\ud W$ is well-defined for all
elements of $\calP_2$. The next two propositions give easily verifiable
conditions for checking that $X\in\calP_2$.

\begin{proposition}\label{P3:Pm-Ext}
Suppose that for some $t>0$ and $p\in [2,+\infty[\:$, a random field
$X=\left\{ X\left(s,y\right): \left(s,y\right)\in \;]0,t[\times \R^d \right\}$
has the following properties:
\begin{itemize}
 \item[(i)] $X$ is adapted and jointly measurable with respect to
$\calB\left(\R^{1+d}\right)\times\calF$;
 \item[(ii)] $\Norm{X(\cdot,\circ) \: 1_{]0,t[}(\cdot)}_{M,p}<+\infty$.
\end{itemize}
Then $X(\cdot,\circ)\: 1_{]0,t[}(\cdot)$ belongs to $\calP_2$.
\end{proposition}

This proposition is taken from \cite[Proposition
\myRef{2.12}{P2:Pm-Ext}]{ChenDalang13Heat}, with $\R$ there replaced by $\R^d$.

\begin{lemma}
\label{L3:Lp}
Let $\calG(s,y)$ be a deterministic measurable function from $\R_+^*\times\R^d$
to $\R$ and let $Z=\left(Z\left(s,y\right):\left(s,y\right)\in
\R_+^*\times \R^d\right)$ be a process such that
\begin{enumerate}[(1)]
 \item $Z$ is adapted and jointly measurable with respect to
$\calB\left(\R^{1+d}\right)\times\calF$,
\item $\Norm{\calG^2(t-\cdot,x-\circ)Z(\cdot,\circ)}_{M,2}<+\infty$
for all $(t,x) \in\R_+^*\times\R^d$.
\end{enumerate}
Then for each $(t,x)\in\R_+\times\R^d$, the random field $\left(s,y\right)\in
[0,t]\times\R^d \mapsto \calG\left(t-s,x-y\right) Z\left(s,y\right)$ belongs to
$\calP_2$ and so the stochastic convolution
\begin{align}
\label{E:StoCon}
\left(\calG\star Z\dot{W}\right)(t,x) := \iint_{[0,t]\times\R^d}
\calG\left(t-s,x-y\right)Z\left(s,y\right)W\left(\ud s,\ud y\right)
\end{align}
is a well-defined Walsh integral and the random field $\calG\star Z\W$ is
adapted.
Moreover, for all even integers $p\ge 2$, and all $(t,x)\in\R_+ \times\R^d$,
\[
\Norm{\left(\calG\star
Z\dot{W}\right)(t,x)}_p^2
\le z_{p}^{2} \Norm{\calG(t-\cdot,x-\circ) Z(\cdot,\circ)}_{M,p}^2.
\]
\end{lemma}

This lemma is taken from \cite[Lemma
\myRef{2.14}{L2:Lp}]{ChenDalang13Heat}, again with $\R$ there replaced by
$\R^d$.

\begin{proposition}\label{P3:Picard}
Suppose that for some even integer $p\in [2,+\infty[\:$, a random field
$Y=\left(Y(t,x): (t,x)\in \R_+^*\times \R^d\right)$
has the following properties
\begin{enumerate}[(i)]
 \item $Y$ is adapted and jointly measurable;
\item for all $(t,x)\in\R_+^*\times\R^d$,
$\Norm{Y(\cdot,\circ)\theta(\cdot,\circ)G(t-\cdot,x-\circ)}_{M,p}^2 <+\infty$.
\end{enumerate}
Then for each $(t,x)\in\R_+^*\times\R^d$,
$Y(\cdot,\circ)\theta(\cdot,\circ)G(t-\cdot,x-\circ)\in\calP_{2}$,
the following Walsh integral
\[
w(t,x)=\iint_{]0,t[\times\R^d}
Y\left(s,y\right)\theta(s,y)G\left(t-s,x-y\right)W(\ud s,\ud y)
\]
is well defined and the resulting random field $w$ is
adapted. Moreover, $w$ is $L^p(\Omega))$-continuous
over $\R_+^*\times\R^d$ under either of the following two conditions:
\begin{enumerate}
\item[(\;\Hcd)] (Heat type):
\begin{enumerate}
 \item[(\;\Hcd-i)] $G$ satisfies Assumptions
\ref{A:Cont-Tail} and \ref{A:Continuous}.
\item[(\;\Hcd-ii)] $\sup_{(t,x)\in K}\Norm{Y(t,x)}_p<+\infty$ for all
compact sets $K \subseteq \R_+^*\times\R^d$,
which is true, in particular, if $Y$ is $L^p(\Omega)$-continuous.
\end{enumerate}
\item[(\;\Wcd)] (Wave type) $G$ satisfies Assumptions \ref{A:Cont-Bdd}.
\end{enumerate}
\end{proposition}

\begin{proof}
Fix $(t,x)\in\R_+^*\times\R^d$. By Assumption (iii) and the fact that
$G(t,x)$ is Borel measurable and
deterministic, the random field
$X=\left(X\left(s,y\right):\:\left(s,y\right)\in\;]0,t[\times\R^d\right)$
with
$X\left(s,y\right):=Y\left(s,y\right)\theta\left(s,y\right)G\left(t-s,
x-y\right)$
satisfies all conditions of Proposition \ref{P3:Pm-Ext}. This implies
that $Y(\cdot,\circ)\theta\left(\cdot,\circ\right)G(t-\cdot,x-\circ)\in
\calP_p$.
Hence $w(t,x)$ is a well-defined Walsh integral and the resulting random field
is adapted to the filtration $\{\calF_s\}_{s\ge 0}$.

Under condition (\Hcd), the proof is identical to that of
\cite[Proposition 2.15]{ChenDalang13Heat}, except that appeals there to
Proposition 2.18 are replaced by appeals to Assumption \ref{A:Cont-Tail}.

Assume condition (\Wcd).
For two points $(t,x), \left(t',x'\right)\in \R_+\times\R^d$, recall $(t_*,x_*)$
and $(\hat{t},\hat{x})$ are defined in \eqref{E:txtx}.
Choose $\beta\in \;]0,1[\:$, $\tau>0$ and $\alpha>0$
according to Assumption \ref{A:Cont-Bdd}.
Fix $(t,x)\in\R_+^*\times\R^d$.
Let $B:=B_{t,x,\beta,\tau,\alpha}$ be the set defined in \eqref{E:B}
and $C$ be the constant used in Assumption \ref{A:Cont-Bdd}.
Assume that $\left(t',x'\right)\in B$.
By Lemma \ref{L3:Lp}, we have that
\begin{align}
 \notag& \Norm{w(t,x)-w\left(t',x'\right)}_p^p \\
&\qquad\le\notag
2^{p-1} z_p^p \left(\int_0^{t_*}\ud s\int_{\R^d} \ud y
 \Norm{Y\left(s,y\right)}_p^2 \:\theta(s,y)^2
\left(G(t-s, x-y)-G(t'-s , x'-y)\right)^2 \right)^{p/2}\\
&\qquad\qquad \notag
+2^{p-1} z_p^p \left( \int_{t_*}^{\hat{t}}\ud s
\int_{\R^d}\ud y\:
\Norm{Y\left(s,y\right)}_p^2\:\theta(s,y)^2
G^2\left(\hat{t}-s,\hat{x}-y\right)\right)^{p/2}\\
&\qquad \le
  2^{p-1} z_p^p \left(L_1(t,t',x,x')\right)^{p/2}
  +
  2^{p-1} z_p^p \left(L_2(t,t',x,x')\right)^{p/2}\;.
\label{E:L1L2}
\end{align}
We first consider $L_1$.
By Assumption \ref{A:Cont-Bdd},
\[
\left(G\left(t-s,x-y\right)-G\left(t'-s,x'-y\right)\right)^2
\le
4 C^2 G^2\left(t+1-s,x-y\right),
\]
and the left-hand side converges pointwise to $0$ for almost all $(t,x)$.
Further,
\begin{align*}
\iint_{[0,t_*] \times \R^{d}}\ud s\ud y &\:4 C^2
 G^2\left(t+1-s,x-y\right)\Norm{Y\left(s,y\right)}_p^2 \: \theta(s,y)^2\\
\le&  4 C^2 \Norm{Y(\cdot,\circ)\theta(\cdot,\circ)
G(t+1-\cdot,x-\circ)}_{M,p}^2,
\end{align*}
which is finite by (ii).
Hence, by the dominated convergence theorem,
\[\lim_{\left(t',x'\right)\rightarrow (t,x)} L_1(t,t',x,x') =0.\]
Similarly, for $L_2$, by Assumption \ref{A:Cont-Bdd},
\[
G^2\left(\hat{t}-s,\hat{x}-y\right) \le C^2 G^2\left(t+1-s,x-y\right).
\]
By the monotone convergence theorem,
$\lim_{\left(t',x'\right)\rightarrow (t,x)} L_2(t,t',x,x') =0$,
because
\begin{align*}
\iint_{\left[t_*,\hat{t}\right]\times \R^{d}} \ud s\ud y &\:C^2
G^2\left(t+1-s,x-y\right)\Norm{Y(s,y)}_p^2\:\theta(s,y)^2
\\
\le &  C^2
\Norm{Y(\cdot,\circ)\theta(\cdot,\circ)G(t+1-\cdot,x-\circ)}_{M,p}^2
\end{align*}
is finite by (ii). This completes the proof under condition (\Wcd).
\end{proof}

We need a lemma which transforms the stochastic integral equation
\eqref{E:WalshSI} into integral inequalities for its moments.
The proof is similar to that of \cite[Lemma 2.19]{ChenDalang13Heat}.

\begin{lemma}\label{L3:HigherMom}
Suppose that $f(t,x)$ is a deterministic
function and $\rho$ satisfies the growth condition \eqref{E:LinGrow}.
If the random fields
$w$ and $v$ satisfy, for all $(t,x)\in\R_+\times\R^d$,
\begin{gather*}
w\left(t,x\right) = f(t,x) +
\left(G \lhd
\left[\rho(v)\dot{W}\right]\right)\left(t,x\right),
\end{gather*}
in which the second term is defined by
\begin{gather*}
 \left(G \lhd
\left[\rho(v)\dot{W}\right]\right)\left(t,x\right)
:= \int_0^t\int_{\R^d}
G\left(t-s,x-y\right)
\theta\left(s,y\right) \;
\rho\left(v\left(s,y\right)\right)
W\left(\ud s,\ud y\right),
\end{gather*}
where we assume that the Walsh integral is well defined,
then for all
even integers $p\ge 2$ and $(t,x)\in\R_+\times\R^d$,
\begin{align*}
 \Norm{\left(G \lhd
\left[\rho(v)\dot{W}\right]\right)\left(t,x\right)}_{p}^2
&\le z_p^2\Norm{G\left(t-\cdot,x-\circ\right) \rho(v(\cdot,\circ)) \:
\theta\left(\cdot,\circ\right)}_{M,p}^2 \\
&\le
\frac{1}{b_p}\left(\left(\Vip^2+\Norm{v}_p^2\right)\rhd
\widehat{\calL}_0\right)\left(t,x\right),
\end{align*}
where $b_p=1$ if $p=2$ and $b_p=2$ otherwise.
In particular,
\[
\Norm{w\left(t,x\right)}_p^2 \le b_p f^2(t,x) +
\left(\left(\Vip^2+\Norm{v}_p^2\right)\rhd
\widehat{\calL}_0\right)\left(t,x\right).
\]
\end{lemma}

\subsubsection{Proof of Theorem \ref{T3:ExUni}}
The proof follows the same six steps as in the proof of \cite[Theorem
2.4]{ChenDalang13Heat} with the
following replacements:

Proposition 2.2 of \cite{ChenDalang13Heat} by Assumptions
\ref{A:Converge}, \ref{A:Gronwall};

Lemma 2.14, ibid., by Lemma \ref{L3:Lp};

Proposition 2.15, ibid., by Proposition \ref{P3:Picard};

Lemma 2.19, ibid., by Lemma \ref{L3:HigherMom};

Lemma 2.21, ibid., by Assumption \ref{A:InDt}.

\noindent
Under Condition \conRef{H}, after making the following further replacements,
the proof will be identical to \cite[Theorem 2.4]{ChenDalang13Heat}:

Proposition 2.16, ibid., by Assumption \ref{A:Continuous} and
Condition \conRef{H}--(a);

Proposition 2.18, ibid., by Assumption \ref{A:Cont-Tail}
and Condition \conRef{H}--(a);

Lemma 2.20, ibid., by Assumption \ref{A:InDt-Bdd} and Condition \conRef{H}--(b).

\noindent
The only care that we should take is that under Condition \conRef{W}, i.e.,
Assumption \ref{A:Cont-Bdd}, the proof should be also modified in certain
places. In the
following, we will highlight these changes.

Recall that in Step 1, we define $u_0(t,x)=J_0(t,x)$ and show by the above
(the first set of) replacements that
\[
I_1(t,x) = \iint_{[0,t]\times\R^d}G(t-s,x-y)\;
\theta\left(s,y\right)\;  \rho\left(u_0\left(s,y\right)\right)
W\left(\ud s,\ud y\right)
\]
is a well defined Walsh integral and the random field
$\left\{I_1\left(t,x\right): (t,x)\in\R_+\times\R^d \right\}$
is adapted and jointly measurable.
The only difference is that the continuity of
$(t,x)\mapsto I_1(t,x)$ from
$\R_+^*\times\R^d$ into $L^p(\Omega)$ is guaranteed by
part (\Wcd) of Proposition \ref{P3:Picard}.

Step 2 gives the Picard iteration, where we assume that for all $k\le n$ and
$(t,x)\in\R_+^*\times\R^d$, the Walsh
integral \[I_k\left(t,x\right) = \iint_{[0,t]\times\R^d}
G\left(t-s,x-y\right)\:
\theta\left(s,y\right)
\:
\rho\left(u_{k-1}\left(s,y\right)\right)
W\left(\ud s,\ud y\right)\]
is well defined such that
\begin{enumerate}[(1)]
 \item $u_k:=J_0+I_k$ is adapted.
 \item The function
$(t,x)\mapsto I_k(t,x)$ from $\R_+^*\times\R^d$ into
$L^p(\Omega)$ is continuous.
 \item $\E\left[u_k^2\left(t,x\right)\right] \!\!=\!\!
J_0^2(t,x)+\sum_{i=0}^{k-1}
\left(\left[\Vip^2+J_0^2\right]
\rhd\calL_i(t,x;\cdot,\circ)\right)\left(t,x\right)$ for the quasi-linear
case and is bounded from above and below (if $\rho$ satisfies
\eqref{E:lingrow} additionally):
\begin{multline*}
 J_0^2(t,x)+
\sum_{i=0}^{k-1}
\left(\left[\vip^2+J_0^2\right]
\rhd\underline{\calL}_i(t,x;\cdot,\circ)\right)\left(t,x\right)
\\ \le
\Norm{u_k(t,x)}_2^2
 \le
J_0^2(t,x)+
\sum_{i=0}^{k-1}
\left(\left[\Vip^2+J_0^2\right]
\rhd\overline{\calL}_i(t,x;\cdot,\circ)\right)\left(t,x\right).
\end{multline*}
\item  $\Norm{u_k\left(t,x\right)}_p^2 \le b_p\:
J_0^2(t,x)+
\sum_{i=0}^{k-1}
\left(\left(\Vip^2 + b_p\:  J_0^2\right)\rhd
\widehat{\calL}_i(t,x;\cdot,\circ)\right)\left(t,x\right)$.
\end{enumerate}
To prove parts (3) and (4) for the case $k=n+1$, we need to apply Lemma
\ref{L3:HigherMom} and \eqref{E:RhdAs2} in
Lemma \ref{L3:RhdAs} to properly deal with the order of the
$\theta$-weighted
convolutions.
Again, the $L^p(\Omega)$-continuity of $(t,x)\mapsto I_{n+1}(t,x)$ is proved by
part (\Wcd) of Proposition \ref{P3:Picard}.

Similarly, in Step 3, we claim that for all $(t,x)\in\R_+^*\times\R^d$, the
series
$\{I_n\left(t,x\right):\:n\in\bbN\}$, with
$I_0\left(t,x\right):=J_0\left(t,x\right)$,  is a Cauchy sequence in
$L^p(\Omega)$. Define $F_n\left(t,x\right)= \Norm{I_{n+1}\left(t,x\right)-
I_n\left(t,x\right)}_p^2$.
For $n\ge 1$, by Lemma \ref{L3:Lp},
\[
F_n\left(t,x\right)\le\left(F_{n-1}\rhd
\widetilde{\calL}_0\right)\left(t,x\right),
\]
where $\widetilde{\calL}_0(t,x) := \calL_0\left(t,x;
z_p\: \max\left(\LIP_\rho, a_{p,\Vip}\Lip_\rho\right)\right)$.
Then apply this relation recursively using  \eqref{E:RhdAs1} in Lemma
\ref{L3:RhdAs} to obtain that
\[F_n(t,x)\le \left(F_{n-1}\rhd
\widetilde{\calL}_0\right)\left(t,x\right)\le
\cdots
\le
\left(
\left(
\cdots
\left(
\left(\left(\Vip^2+J_0^2\right) \rhd \widetilde{\calL}_0\right)\rhd
\widetilde{\calL}_0\right)
\rhd \cdots \right)
\rhd \widetilde{\calL}_0
\right)
\left(t,x\right),
\]
where the r.h.s. of the inequality has $n+1$ convolutions.
We now apply \eqref{E:RhdAs1} in Lemma
\ref{L3:RhdAs}. then Assumption \ref{A:Gronwall} to
obtain
\[F_n\left(t,x\right)\le
\left(\left[ \Vip^2+J_0^2\right] \rhd
\widetilde{\calL}_n(t,x;\cdot,\circ)\right)\left(t,x\right)
\le
\left(\left[ \Vip^2+J_0^2\right] \rhd
\widetilde{\calL}_0\right)\left(t,x\right)B_n(t),
\]
where the kernel functions $\widetilde{\calL}_n\left(t,x;s,y\right)$ are
defined by the same parameter as $\widetilde{\calL}_0(t,x)$.

Towards the end of Step 4, we need to apply Lebesgue's dominated convergence
theorem. To check the integrability of the integrand, we use \eqref{E:KUP} and
then Lemma \ref{L3:RhdAs}.

In Step 5, when we convolve an extra kernel function $\widetilde{\calK}$, again
we need to apply \eqref{E:RhdAs10} in Lemma \ref{L3:RhdAs} to deal with
the order of the $\theta$--weighted convolution.

With these replacements and changes, Theorem \ref{T3:ExUni} is also proved under
Condition \conRef{W}.
\myEnd

\subsection{Application to the stochastic heat equation with distribution-valued
initial data}
\label{SS:ThetaHeat}

We apply Theorem \ref{T3:ExUni} to study the stochastic heat equation
\begin{align}\label{E1:ThetaHeat}
\begin{cases}
\left(\frac{\partial }{\partial t} - \frac{\nu}{2}
\frac{\partial^2 }{\partial x^2}\right) u(t,x) =  \rho(u(t,x))
\:\theta(t,x)\:\dot{W}(t,x),&
x\in \R,\; t \in\R_+^*\;, \\
\quad u(0,\cdot) = \mu(\cdot)\;.
\end{cases}
\end{align}
Let $G_\nu(t,x)$ be the heat kernel, i.e.,
\begin{align}\label{E:HeatG}
G_\nu(t,x)=\frac{1}{\sqrt{4\pi\nu
t}}\exp\left(-\frac{x^2}{4t} \right),
\quad\text{for $t>0$ and $x\in\R$.}
\end{align}
We will focus on this equation with general initial data, and we will study
how certain properties of $\theta(t,x)$ function affect the {\it admissible
initial data} -- the initial data starting from which the stochastic
heat equation \eqref{E1:ThetaHeat} admits a random field solution.
Recall that  \cite[Proposition \myRef{2.11}{P2:D-Delta}]{ChenDalang13Heat}
shows that if $\theta(t,x)\equiv 1$, then the initial data cannot go beyond
measures.

As for the properties of $\theta(t,x)$, we will not pursue the full
generality here. Instead, we only consider certain particular
$\theta(t,x)$ to show the balance between certain properties of $\theta(t,x)$
and the set of the admissible initial data.
For $r\ge 0$, define
\[
\Xi_r :=\left\{\theta: \R_+\times\R \mapsto \R:  \sup_{(t,x)\in
\R_+^*\times\R} \frac{\left|\theta(t,x)\right|}{t^r\wedge
1}<+\infty\right\},\;\;\text{and}\;\;
\Xi_{\infty} := \bigcap_{n\in \bbN} \Xi_n\:.
\]
Clearly, if $0\le m\le n$, then $\Xi_m\supseteq \Xi_n$. Here are some simple
examples: $ t^k\wedge 1 \in \Xi_k$ for all $k\ge 0$;
$\exp\left(-1/t\right)\in \Xi_{+\infty}$.

Let $C^\infty_c(\R)$ be the space of  the $C^\infty$-functions with compact
support.
Let $\calD'(\R)$ be the space of distributions --- the dual space of
$C_c^\infty(\R)$. Let $\mu$ be a locally finite measure on $\R$
and let $\mu=\mu_+ - \mu_-$ be its Jordan decomposition into
two non-negative measures with disjoint supports.
Denote $| \mu|=\mu_++\mu_-$.

\begin{definition}\label{D2:D'k}
Let $\calM_H(\R)$ be the set of signed Borel measures $\mu$ on $\R$ such that
for all $t>0$ and $x\in\R$, $\left(|\mu|*G_\nu(t,\cdot)\right)(x)<+\infty$.
For $k\in\bbN$, define
\[\calD_k'\left(\R\right)= \left\{\mu\in \calD'\left(\R\right): \exists
\mu_0\in
\calM_H\left(\R\right), \text{s.t.}\; \mu = \mu_0^{(k)}
\right\},\quad\text{and}\quad
\calD_{+\infty}'\left(\R\right)=\bigcup_{k\in\bbN}
\calD_k' \left(\R\right),
\]
where $\mu_0^{(k)}$ denotes the $k$-th distributional
derivative.
\end{definition}

\begin{theorem} \label{T3:ThetaHeat}
Suppose that $\rho$ is Lipschitz continuous.
If $\theta(t,x)\in \Xi_r$ for some $0\le r\le +\infty$, then
\eqref{E1:ThetaHeat} has a solution
$\left\{u(t,x):\; t>0,x\in\R\right\}$ in the sense of Definition
\ref{D3:Solution}
for any initial data $\mu \in \calD_k'\left(\R\right)$ with $k\in\bbN$ and $0\le
k<2r+1/2$.
Moreover, the solution $u(t,x)$ is unique (in the sense of
versions) and is $L^p(\Omega)$ -continuous over
$\R_+^*\times\R$ for all $p\ge 2$.
In addition, the estimates of Theorem \ref{T3:ExUni} apply.
\end{theorem}

The proof of this theorem is given at the end of this section.

\begin{example}
If $\theta(t,x)\equiv 1$, then $\theta \in \Xi_r$ if and only if $r=0$. So, by
Theorem \ref{T3:ThetaHeat}, the admissible initial
data are $\calD_0'\left(\R\right)$, which recovers the condition
$\left(|\mu|* G_\nu(t,\cdot)\right)(x)<\infty$ for all $t>0$ and $x\in\R$ in
\cite{ChenDalang13Heat}.
\end{example}

\begin{example}[Derivatives of the Dirac delta functions]
If $\theta(t,x)=t^r \wedge 1$, then the initial data can be $\delta^{(k)}_0$
with
$0\le k<2r+1/2$. This is consistent with \cite[Proposition
\myRef{2.11}{P2:D-Delta}]{ChenDalang13Heat}. If
$\theta(t,x)=\exp\left(-1/t\right)$,
then all
derivatives of $\delta_0$ are admissible initial data.
\label{e2.25}
\end{example}

\begin{example}[Schwartz distribution-valued initial data and beyond]
If we choose $\theta(t,x)\in \Xi_{+\infty}$, for example $\theta(t,x)=
\exp\left(-1/t\right)$, then the initial data can be any Schwartz distribution.
Actually, the admissible initial
data $\calD_{+\infty}'\left(\R\right)$ can go beyond Schwartz distributions.
Here are some
simple examples: $\mu(\ud x) = \mu_0^{(k)}(\ud x)$ for any $k\in\bbN$, where $
\mu_0(\ud x) =e^{|x|}\ud x$.
\label{e2.26}

\end{example}


Let $\partial_y^n$ and $\partial_t^n$  be the $n$-th
partial derivatives with respect to $y$ and $t$, respectively. In particular,
\[\partial_y^k \left[G_\nu\left(t,x-y\right)\right] =
(-1)^k\left.\frac{\partial^k }{\partial z^k} G_\nu(t,z)\right|_{z= x-y}
=
(-1)^k \partial_x^k G_\nu\left(t,x-y\right).
\]
As a special case of a standard result (see, e.g., \cite[Theorem
1, Chapter 9, p.241]{FriedmanPDE} or \cite[(15), p. 15]{Eidelman69PS}), for all
$t\ge 0$ and
$n\in\bbN$, there are two constants $C_n$ and $\nu_n$ depending
only\footnote{There is no dependence on a finite horizon $T>0$ because the
coefficients of our parabolic equation are constant, while in
both \cite{Eidelman69PS} and \cite{FriedmanPDE} they are time-dependent. See
Remark \ref{R:Hermite} for a brief proof of this fact.} on $n$
and $\nu$ such that
\begin{align}\label{E:DifG}
\partial_y^n\: G_\nu(t,x-y) \le \frac{C_n}{t^{n/2}} G_{\nu_n}(t,x-y)\;, \quad
\text{for all $t\ge 0$, and $x,y\in\R$}.
\end{align}

\begin{remark}\label{R:Hermite}
For the heat kernel function,
the bound in \eqref{E:DifG} can be improved.
Let $\He_n(x;t)$ be the {\it Hermite polynomials}:
 \[\He_n\left(x\:;\:t\right) := \sum_{k=0}^{\Floor{n/2}} \binom{n}{2k}
(2k-1)!!\:
\left(-\frac{x}{\sqrt{t}}\right)^{n-2k},\quad\text{for all $t>0$ and $x\in\R$,}
\]
where $\Floor{n/2}$ is the largest integer not bigger than $n/2$ and
$n!!$ is the double factorial (see \cite{NIST2010}).
Then $\partial_y^n \left[ G_\nu\left(t,x-y\right)\right]
 = (\nu t)^{-n/2} G_\nu\left(t,x-y\right) \He_n\left(x-y;\nu
t\right)$; see Theorem 9.3.3 of \cite{Kuo05Introduction}. Then one can remove
the Hermite polynomials by increasing the parameter $\nu$ in the heat kernel
function to obtain the upper bound of the form \eqref{E:DifG}.
\end{remark}

\begin{lemma}\label{L2:SwitchDI}
Suppose that $\mu\in\calM_H(\R)$, and $n,m,a,b\in\bbN$.
Then for all $t>0$ and $x\in\R$,
\[\partial_t^a\partial_x^b \int_\R \mu(\ud y) \:\partial_t^n \partial_x^m
G_\nu\left(t,x-y\right)
=\int_\R \mu(\ud y)\: \partial_t^{n+a} \partial_x^{m+b} G_\nu\left(t,x-y\right).
\]
\end{lemma}
Note that $\partial_t\: G_\nu = \nu/2\: \partial^2_x \:G_\nu$.
The proof consists of using standard results (e.g.,
\cite[Theorem 16.8]{Billingsley95Book}) on permuting integrals and
differential signs. Now define
\begin{align}\label{E2:DefineJ0}
J_0(t,x)&:=
(-1)^k\left(\mu_0*  \partial_y^{k} \left[G_1( \nu t,\cdot)\right]\right)(x),\;
\text{for all $(t,x)\in\R_+^*\times \R$}\;,
\end{align}
which, by \eqref{E:DifG}, can be bounded by,
\begin{align}\label{E2:HomeSlt-Eqv}
\left|J_0(t,x)\right|\le
C_k t^{-k/2}
\left(|\mu_0|* G_{\nu_k}(t,\cdot)\right)(x)\;,
\end{align}
for some positive constants $C_k$ and $\nu_k$.
As a direct consequence of Lemma \ref{L2:SwitchDI}, for all
$\mu\in\calD_k'\left(\R\right)$, $J_0(t,x)$ defined in
\eqref{E2:DefineJ0} belongs to $C^\infty\left(\R_+^*\times\R\right)$, which
is the smoothing property of the heat kernel.

The following lemma is a standard result (see \cite{friedman} and also
\cite[Proposition 2.6.14]{LeChen13Thesis}).

\begin{lemma}\label{L2:HomeSlt}
Suppose that $\mu\in \calD_k'(\R)$, $k\in\bbN$. Let $\mu_0\in\calM_H(\R)$ be
the signed Borel measure associated to $\mu$ such that $\mu=\mu_0^{(k)}$.
Then the function $J_0(t,x)$ defined in \eqref{E2:DefineJ0} solves
\begin{align}\label{E2:Heat-home}
\begin{cases}
\left(\frac{\partial }{\partial t} - \frac{\nu}{2}
\frac{\partial^2 }{\partial x^2}\right) u(t,x) =  0,&
x\in \R,\; t \in\R_+^*, \\
\quad u(0,\cdot) = \mu(\cdot)\;,
\end{cases}
\end{align}
and
$\lim_{t\rightarrow 0_+}\InPrd{\psi, J_0(t,\cdot)} = \InPrd{\psi,\mu}$
for all $\psi\in C_c^\infty(\R)$.
\end{lemma}

\begin{proposition}\label{P3:InDt}
Suppose that $\theta(t,x)\in \Xi_r$ and $\mu\in \calD_k'\left(\R\right)$ with
$0\le k <2r+1/2$. Then for all $v>0$ and all compact sets
$K \subseteq \R_+^*\times\R$,
\[\sup_{(t,x)\in K}\left(\left[v^2 + J_0^2\right]\rhd G_\nu^2
\right)(t,x)<+\infty.\]
\end{proposition}
\begin{proof}
Let $\mu_0\in\calM_H(\R)$ be such that $\mu=\mu_0^{(k)}$.
Then $J_0(t,x)$ given in \eqref{E2:DefineJ0}
is a weak solution to the homogeneous equation (see also \cite[Lemma
2.6.14]{LeChen13Thesis}).
We assume first that $v=0$.
Since for some constant $C$, $|\theta(t,x)|  \le C
\left(1\wedge t^r\right)\le C t^r$,
it suffices to prove that, for all compact sets $K\subseteq\R_+^*\times\R$,
\[
\sup_{(t,x)\in K} f(t,x)<+\infty,\quad \text{where}\;
f(t,x):= \iint_{[0,t]\times \R}  \ud s\ud y\: J_0^2\left(s,y\right) s^{2r}
G_\nu^2\left(t-s,x-y\right).
\]
Without loss of generality, we assume from now that the measure $\mu_0$ is
non-negative. We will use the bound on $J_0(t,x)$ in \eqref{E2:HomeSlt-Eqv} and
denote $\xi:=\nu_k$.
Because $\xi >\nu$ (see Remark \ref{R:Hermite}),
\[
\sup_{(s,y)\in[0,t]\times\R}\; \frac{G_\nu(t-s,x-y)}{G_\xi(t-s,x-y)}<+\infty\:.
\]
Hence, for some constant $C>0$,
\[
|f(t,x)|\le C \iint_{[0,t]\times \R} \ud s\ud y\:
s^{2r-k}\left(\mu_0*G_\xi(s,\cdot)\right)^2(y) \; G_\xi^2(t-s,x-y).
\]
Then write $\left(\mu_0*G_\xi(s,\cdot)\right)^2(y)$ in the form of double
integral and use Lemma \ref{L2:GG}:
\begin{align*}
 \left| f(t,x) \right| \le
\int_0^t\ud s\: \frac{C \:s^{2r-k}}{\sqrt{4\pi\xi (t-s)}}&
\iint_{\R^2}
\mu_0(\ud z_1)\mu_0(\ud z_2)\:
G_{2\xi}\left(s,z_1-z_2\right)
\\
&\times \int_\R \ud y\: G_{\frac{\xi}{2}}\left(s,y-\bar{z}\right)
G_{\frac{\xi}{2}}\left(t-s,x-y\right),
\end{align*}
where $\bar{z}=(z_1+z_2)/2$.
By the semigroup property of the heat kernel function,
\[
 \left| f(t,x) \right| \le
\int_0^t\ud s\: \frac{C \:s^{2r-k}}{\sqrt{4\pi\xi (t-s)}}
\iint_{\R^2}
\mu_0(\ud z_1)\mu_0(\ud z_2)\:
G_{2\xi}\left(s,z_1-z_2\right) \: G_{\frac{\xi}{2}}(t,x-\bar{z}).
\]
Apply Lemma \ref{LH:Split} to $G_{2\xi}\left(s,z_1-z_2\right) \:
G_{\frac{\xi}{2}}(t,x-\bar{z})$ to see that
\begin{align}\label{E_:HeatInt}
 \left| f(t,x) \right| \le
\left(\mu_0 * G_{2\xi}(t,\cdot)\right)^2\!(x)\:
\int_0^t\ud s\: \frac{C \:s^{2r-k-1/2}\:\sqrt{t}}{\sqrt{\pi\xi
(t-s)}} \:.
\end{align}
The integration over $s$ is finite since $2r-k-1/2> -1$.
By the smoothing effect of the heat kernel, for any arbitrary compact
set $K\subseteq\R_+^*\times\R$,  $\sup_{(t,x)\in K}\left(\mu_0 *
G_{2\xi}(t,\cdot)\right)^2\!(x)$ is finite. This proves the proposition with
$v=0$.
As for the contribution of $v$, we simply replace $\mu_0(\ud x)$ by $v\:\ud x$
in \eqref{E_:HeatInt}.
This completes the proof of Proposition  \ref{P3:InDt}.
\end{proof}

\begin{proof}[Proof of Theorem \ref{T3:ThetaHeat}]
We only need to verify that Conditions \conRef{G} and
\conRef{H} of Theorem \ref{T3:ExUni} are satisfied.
Fix $r\in [0,+\infty]$ and $\theta(t,x)\in \Xi_r$. Since $\theta$ is uniformly
bounded and $d=1$, Assumption \ref{A:R-G} is satisfied.
Assumptions \ref{A:Converge} and \ref{A:Gronwall} are verified
by
\cite[Proposition \myRef{2.2}{P2:K}]{ChenDalang13Heat} with $\lambda=C
\Lip_\rho$.
Assumption \ref{A:InDt} is true due to Proposition \ref{P3:InDt}, where the
hypothesis $0\le k<2r+1/2$ is used.
Therefore, all conditions in \conRef{G}
are satisfied.
Both Assumptions \ref{A:Cont-Tail} and \ref{A:Continuous} are
satisfied due
to Propositions \myRef{2.18}{P2:G-Margin} and \myRef{2.16}{P2:G} of
\cite{ChenDalang13Heat},
respectively. Assumption \ref{A:InDt-Bdd} is true by
Lemma \myRef{2.20}{L2:J0Cont}, ibid.
Therefore, all conditions in \conRef{H} are satisfied.
This completes the proof of Theorem \ref{T3:ThetaHeat}.
\end{proof}

\section{Stochastic wave equation}
\label{S:MR-Wave}
We now turn to the study of the stochastic wave equation \eqref{E4:Wave}.
Recall the formulas for $J_0(t,x)$ and for the fundamental solution
$G_\kappa(t,x)$ given in \eqref{E4:J0-Wave}.

\subsection{Existence, uniqueness, moments and regularity}
\label{SS:Wave-Result}
Define a kernel function
\begin{align}\label{E4:BesselI}
\calK\left(t,x;\:\kappa,\: \lambda\right):=
\begin{cases}
\frac{\lambda^2}{4} I_0\left(\sqrt{\frac{\lambda^2\left((\kappa
t)^2-x^2\right)}{2 \kappa}}\right)& \text{if
$-\kappa t\le x\le \kappa t$\;,}\cr
0&\text{otherwise}\;,
\end{cases}
\end{align}
with two parameters $\kappa>0$ and $\lambda>0$,
where $I_n(\cdot)$ is the modified Bessel function of the first kind of order
$n$, or simply the {\em hyperbolic Bessel function} (\cite[10.25.2, on p.
249]{NIST2010}):
\begin{align}\label{E4:In}
 I_n(x) := \left(\frac{x}{2}\right)^n  \sum_{k=0}^{\infty}
\frac{\left(x^2/4\right)^k}{k! \: \Gamma(n+k+1)}\;.
\end{align}
See \cite{Kevorkian99PDE-V2,Watson58Bessel} for its
relation  with the wave equation.
Define
\begin{align}\label{E4:Sinh}
\calH\left(t;\:\kappa,\lambda\right) :=
\left(1\star\calK\right) (t,x)=
\:
\cosh\left(|\lambda|\sqrt{\kappa/2}\: t\right)-1\;,
\end{align}
where the second equality is proved in Lemma \ref{L:Ht} below. The following
bound on $I_0(x)$ will be useful and convenient for the later applications of
the moment formula:
\begin{align}\label{E4:I0Exp}
I_0(z) \le \cosh(z) \le  e^{|z|},\quad \text{for all $z\in\R$},
\end{align}
which can be seen from the formula $I_0(z)=\frac{1}{\pi}\int_0^\pi \ud \theta\:
\cosh(z \cos(\theta))$ (see \cite[(10.32.1)]{NIST2010}).
We use the same conventions as \eqref{E:Convention} regarding to the parameter
$\lambda$. For example, $\calK(t,x):= \calK\left(t,x;\kappa,
\lambda\right)$
and $\widehat{\calK}_p(t,x):=
\calK\left(t,x;\kappa, a_{p,\Vip} z_p \Lip_\rho \right)$.
Define two functions:
\begin{gather}\label{E4:Ttx}
 T_\kappa(t,x) := \left(t - \frac{|x|}{2\kappa}\right) \; \Indt{|x|\le 2
\kappa t}\;,
\\
\Theta_\kappa\left(t,x,y\right):= \iint_{\R_+\times\R}\ud s \ud z\:
G_\kappa(t-s,x-z)G_\kappa\left(t-s,y-z\right)= \frac{\kappa}{4}
T^2_\kappa
\left(t,x-y\right),\label{E4:Theta}
\end{gather}
where the second equality is proved in Lemma \ref{L:Theta}.
This is the quantity 
$\Theta\left(t,x,y\right)$ in \eqref{E:Theta}.
Let $\calM\left(\R\right)$ be the set of locally finite (signed) Borel measures
over $\R$.
\begin{theorem}\label{T4:ExUni}
Suppose that $g\in L_{loc}^2\left(\R \right)$,
$\mu \in \calM \left(\R\right)$ and $\rho$ is Lipschitz continuous with
$|\rho(u)|^2 \le \Lip_\rho^2\left(\Vip^2+u^2\right)$.
Define $\overline{\calK}$, $\overline{\calH}$, $T_\kappa$, etc., as above.
Then the stochastic integral equation \eqref{E4:WaveInt} has a random field
solution, in the sense of Definition \ref{D3:Solution}:
$u(t,x)=J_0(t,x)+I(t,x)$ for $t>0$ and $x\in\R$. Moreover,\\
(1) $u(t,x)$ is unique (in the sense of versions);\\
(2) $(t,x)\mapsto I(t,x)$ is $L^p(\Omega)$-continuous for all integers $p\ge
2$;\\
(3) For all even integers $p\ge 2$ and all $t>0$, $x,y \in\R$,
\begin{gather}
\label{E4:SecMom-Up}
\Norm{u(t,x)}_p^2 \le
\begin{cases}
J_0^2(t,x) + \left(J_0^2 \star \overline{\calK} \right) (t,x) +
\Vip^2  \overline{\calH}(t)& \text{if $p=2$,}\cr
2J_0^2(t,x) + \left(2J_0^2 \star \widehat{\calK}_p \right) (t,x) +
\Vip^2  \widehat{\calH}_p(t)&\text{if $p>2$,}
\end{cases}\\
\label{E4:TP-Up}
\E\left[u(t,x)u(t,y)\right]  \le
J_0(t,x)J_0(t,y) + 
 \frac{\kappa\Lip_\rho^2 \Vip^2}{4}
T^2_\kappa
\left(t,x-y\right)
+\frac{\Lip_\rho^2}{2}
\left(f\star G_\kappa\right)\left(T,\frac{x+y}{2}\right),
\end{gather}
where $T=T_\kappa\left(t,x-y\right)$ and $f(s,z)$ denotes the r.h.s. of
\eqref{E4:SecMom-Up} for
$p=2$;\\
(4) If $\rho$ satisfies \eqref{E:lingrow}, then
for all $t>0$, $x,y\in\R$,
\begin{gather}
\label{E4:SecMom-Lower}
\Norm{u(t,x)}_2^2 \ge J_0^2(t,x) + \left(J_0^2 \star \underline{\calK} \right)
(t,x)+
\vip^2\: \underline{\calH}(t),\\
\E\left[u(t,x)u(t,y)\right] \ge
J_0(t,x)J_0(t,y) + \frac{\kappa \lip_\rho^2\: \vip^2}{4}\:
T^2_\kappa\left(t,x-y\right)
+\frac{\lip_\rho^2}{2}
\left(f\star
G_\kappa\right)\left(T,\frac{x+y}{2}\right),
\label{E4:TP-Lower}
\end{gather}
where $T=T_\kappa\left(t,x-y\right)$ and $f(s,z)$ denotes the r.h.s. of
\eqref{E4:SecMom-Lower}; \\
(5) In particular, if $|\rho(u)|^2=\lambda^2
\left(\vv^2+u^2\right)$, then for all $t>0$, $x,y\in\R$,
\begin{gather}
 \label{E4:SecMom}
 \Norm{u(t,x)}_2^2 = J_0^2(t,x) + \left(J_0^2 \star \calK \right) (t,x)+ \vv^2
\: \calH(t),\\
\E\left[u(t,x)u(t,y)\right] =
J_0(t,x)J_0(t,y) +  \frac{\kappa \lambda^2 \vv^2}{4} \:
T^2_\kappa\left(t,x-y\right)
+ \frac{\lambda^2}{2}
\left(f\star
G_\kappa\right)\left(T,\frac{x+y}{2}\right),
\label{E4:TP}
\end{gather}
where $T=T_\kappa\left(t,x-y\right)$ and $f(s,z)=\Norm{u(s,z)}_2^2$ is defined
in \eqref{E4:SecMom}.
\end{theorem}

The proof of this theorem is given at the end of Section \ref{ss:wave-lp}.

\begin{corollary}[Constant initial data]\label{C4:HomeInit}
Suppose that $\rho^2(x)=\lambda^2(\vv^2+x^2)$ with $\lambda\ne 0$.
Let $\calH(t)$ be defined as above.
If $g(x)\equiv w$ and $\mu(\ud x) = \widetilde{w}\; \ud x$ with
$w,\widetilde{w}\in \R$, then:
\begin{enumerate}[(1)]
 \item For all $t\ge 0$ and $x\in\R$,
\[\Norm{u(t,x)}_2^2 =
w^2
+ \left(w^2 + \vv^2+ \frac{4\kappa
\widetilde{w}^2}{\lambda^2}\right)\calH(t)
+\frac{2\sqrt{2\kappa} w\widetilde{w}}{|\lambda|}
\sinh\left(\frac{\sqrt{\kappa} |\lambda| t}{\sqrt{2}}\right).\]
In particular,  \[\Norm{u(t,x)}_2^2 =
\begin{cases}
w^2 \left(\calH\left(t\right)+1\right)
& \text{if $\vv=\widetilde{w}=0$,}\\[0.5em] \displaystyle
 \frac{4\kappa
\widetilde{w}^2}{\lambda^2}
\; \calH(t)
&\text{if $\vv=w=0$.}
\end{cases}\]
\item For all $t\ge 0$ and $x,y\in\R$, set $T=T_ \kappa(t,x-y)$. Then
\begin{multline*}
\E\left[u(t,x)u(t,y)\right]=w^2 + \kappa \widetilde{w}
\left(t-T\right)
\left(2w
+\kappa\widetilde{w}(t+T)\right)
\\+ \left(w^2 + \vv^2+ \frac{4\kappa
\widetilde{w}^2}{\lambda^2}\right)
\calH\left(T\right)
+\frac{2\sqrt{2\kappa} w\widetilde{w}}{|\lambda|}
\sinh\left(\frac{\sqrt{\kappa}
|\lambda|}{\sqrt{2}}T\right).
\end{multline*}
In particular,
\[\E\left[u(t,x)u(t,y)\right] =
\begin{cases}
w^2 \left(\calH\left(T\right)+1\right)
& \text{if $\vv=\widetilde{w}=0$,}\\[0.5em]  \displaystyle
 \frac{4\kappa
\widetilde{w}^2}{\lambda^2}\;
\calH\left(T\right)
+\kappa^2\widetilde{w}^2\left(t^2-T^2\right)
&\text{if $\vv=w=0$.}
\end{cases}\]
\end{enumerate}
\end{corollary}

\begin{proof}
(1) In this case, $J_0(t,x)= w+ \kappa \widetilde{w} t$.
The formula for $\Norm{u(t,x)}_2^2$ follows from the moment
formula \eqref{E4:SecMom} and the integrals in Lemmas \ref{L:Ht}
and \ref{L:IntShCh}.

(2) The formulas follow from \eqref{E4:TP} and (1), and the integrals in
\eqref{E4:Theta} and Lemma \ref{L:IntShCh}.
\end{proof}

\begin{corollary}[Dirac delta initial velocity]\label{C4:DeltaInit}
Suppose that $\rho^2(x)=\lambda^2(\vv^2+x^2)$ with $\lambda\ne 0$.
If $g\equiv 0$ and $\mu=\delta_0$, then for all $t\ge 0$ and $x,y\in\R$,
\[\E\left[u(t,x)u\left(t,y\right)\right]
=
\frac{1}{\lambda^2}\calK\left(T_\kappa\left(t,x-y\right),\frac{x+y}{2}\right)
+
\vv^2 \calH\left(T_\kappa\left(t,x-y\right)\right).
\]
In particular, $\Norm{u(t,x)}_2^2 = \frac{1}{\lambda^2} \calK(t,x) +
\vv^2 \calH(t)$.
\end{corollary}

\begin{proof}
In this case, $J_0(t,x)=G_\kappa(t,x)$ and so $\lambda^2 J_0^2(t,x) =
\calL_0(t,x)$.
Set $T=T_\kappa(t,x-y)$ and $\bar{x}=(x+y)/2$.
By \eqref{E4:SecMom} and Proposition \ref{P4:K-W}, $\Norm{u(t,x)}_2^2 =
\frac{1}{\lambda^2}\calK(t,x)+\vv^2 \calH(t)$. By \eqref{E4:TP} and
\eqref{E4:GG},
\begin{align*}
\E\left[u(t,x)u\left(t,y\right)\right] = &\frac{1}{2}
G_\kappa\left(T,\bar{x}\right) +
\lambda^2\vv^2 \Theta_\kappa\left(t,x,y\right)
\\&+ \frac{\lambda^2}{2} \int_0^{T}\ud
s\int_\R \ud z\:
\left(\frac{1}{\lambda^2}\calK(s,z) + \vv^2\calH(s) \right)
G_{\kappa}\left(T-s,\bar{x}-z\right).
\end{align*}
By \eqref{E4:K-K0}, the double integral with $\lambda^2/2$ in
the above formula equals
\[\frac{1}{\lambda^2}\calK\left(T,\bar{x}\right)
-\frac{1}{2} G_\kappa\left(T,\bar{x}\right) + I,
\]
where \[I=\frac{\lambda^2\vv^2}{2} \int_0^T  \ud s\: \calH(s)\int_\R\ud z\:
G_{\kappa}\left(T-s,\bar{x}-z\right).\]
Now let us evaluate the integral $I$. The $\ud
z$--integral is equal to $\kappa(T-s)$. By \eqref{E4:Sinh} and Lemma
\ref{L:IntShCh},
\[I = \frac{\lambda^2\vv^2}{2}\int_0^{T}\ud s\:
\calH(s) \; \kappa \left(T-s\right)=
\vv^2 \calH\left(T\right) -\frac{\kappa \lambda^2\vv^2
 }{4}T^2=
\vv^2 \calH\left(T\right) -\lambda^2\vv^2
\Theta_\kappa\left(t,x,y\right).\]
Finally, the corollary is proved by combining these terms.
\end{proof}

\begin{example}\label{Ex4:wave-g-1/4}
Let $g(x)= |x|^{-1/4}$ and $\mu \equiv 0$. Clearly, $g\in
L_{loc}^2\left(\R\right)$ and
\[J_0^2(t,x) = \frac{1}{4}\left(\frac{1}{|x+\kappa t|^{1/4}}+\frac{1}{|x-\kappa
t|^{1/4}}\right)^2.\]
The function $J_0^2(t,x)$ equals $+\infty$ on the
characteristic lines $x=\pm \kappa t$ that originate at $(0,0)$, where the
singularity of $g$ occurs. Nevertheless,
the stochastic integral part $I(t,x)$ is well defined for all
$(t,x)\in\R_+^*\times\R$ and the random field solution $u(t,x)$ in
the sense of Definition \ref{D3:Solution} does
exist according to Theorem \ref{T4:ExUni}. We note that the argument for the
heat equation in Theorem \ref{T3:ExUni}, which is
based on Condition \conRef{H}, cannot be used here because of the singularity of
$J_0(t,x)$ at certain points.
However, the wave kernel function satisfies Condition \conRef{W}, which is not
satisfied by the heat kernel.
\end{example}

\begin{figure}[h!tbp]
\center
\includegraphics[scale=0.85]{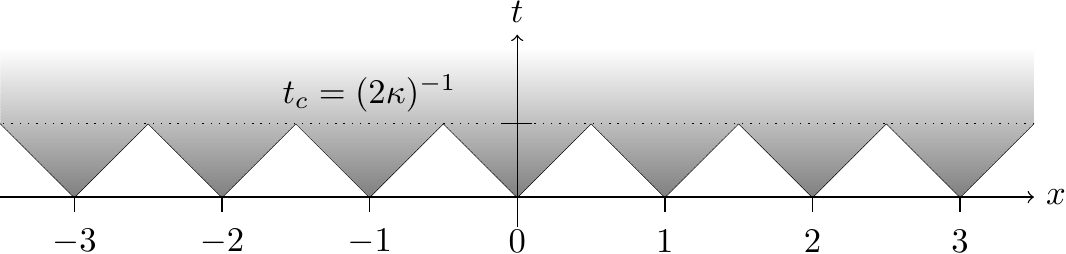}
\caption{When $g(x)= \sum_{n\in
\bbN}2^{-n}\left(|x-n|^{-1/2}+|x+n|^{-1/2}\right)$ and $\mu \equiv 0$,
the random field solution $u(t,x)$ is only defined
in the unshaded regions and in particular not for $t>t_c =(2\kappa)^{-1}$.}
\label{F4:wave-g-1/2-N}
\end{figure}

\begin{example}\label{Ex4:wave-g-1/2}
Let $g(x)= |x|^{-1/2}$ and $\mu \equiv 0$. Clearly, $g\not\in
L_{loc}^2\left(\R\right)$. So Theorem \ref{T4:ExUni} does not apply. In this
case,
the solution $u(t,x)$ is well defined outside of the triangle $\kappa\: t\ge
|x|$. But because
\[J_0^2(t,x) = \frac{1}{4}\left(\frac{1}{|x+\kappa t|^{1/2}}+\frac{1}{|x-\kappa
t|^{1/2}}\right)^2,\]
and this function is not locally integrable over domains that intersect the
characteristic lines $x=\pm \kappa t$ (see Assumption \ref{A:InDt}),
the random field solution exists only in the two ``triangles'' $\kappa \: t\le
|x|$.
Another
example is shown in Figure \ref{F4:wave-g-1/2-N}.
\end{example}

\subsection{Some lemmas and propositions for the existence theorem}
\label{ss:wave-lp}
Define the backward space-time cone:
\[\Lambda(t,x)= \left\{\left(s,y\right)\in\R_+\times\R:\; 0\le s\le t,
\;\;|y-x|\le
\kappa (t-s) \right\}\]
and the wave kernel function can be equivalently written
as $G_\kappa\left(t-s,x-y\right) =
\frac{1}{2}1_{\Lambda(t,x)}\left(s,y\right)$.
The change of variables $u=\kappa s-y$, $w=\kappa s +y$ will play an important
role: see Figure \ref{F4:ChangeVar}.

\begin{figure}[h!tbp]
\center
\includegraphics[scale=0.75]{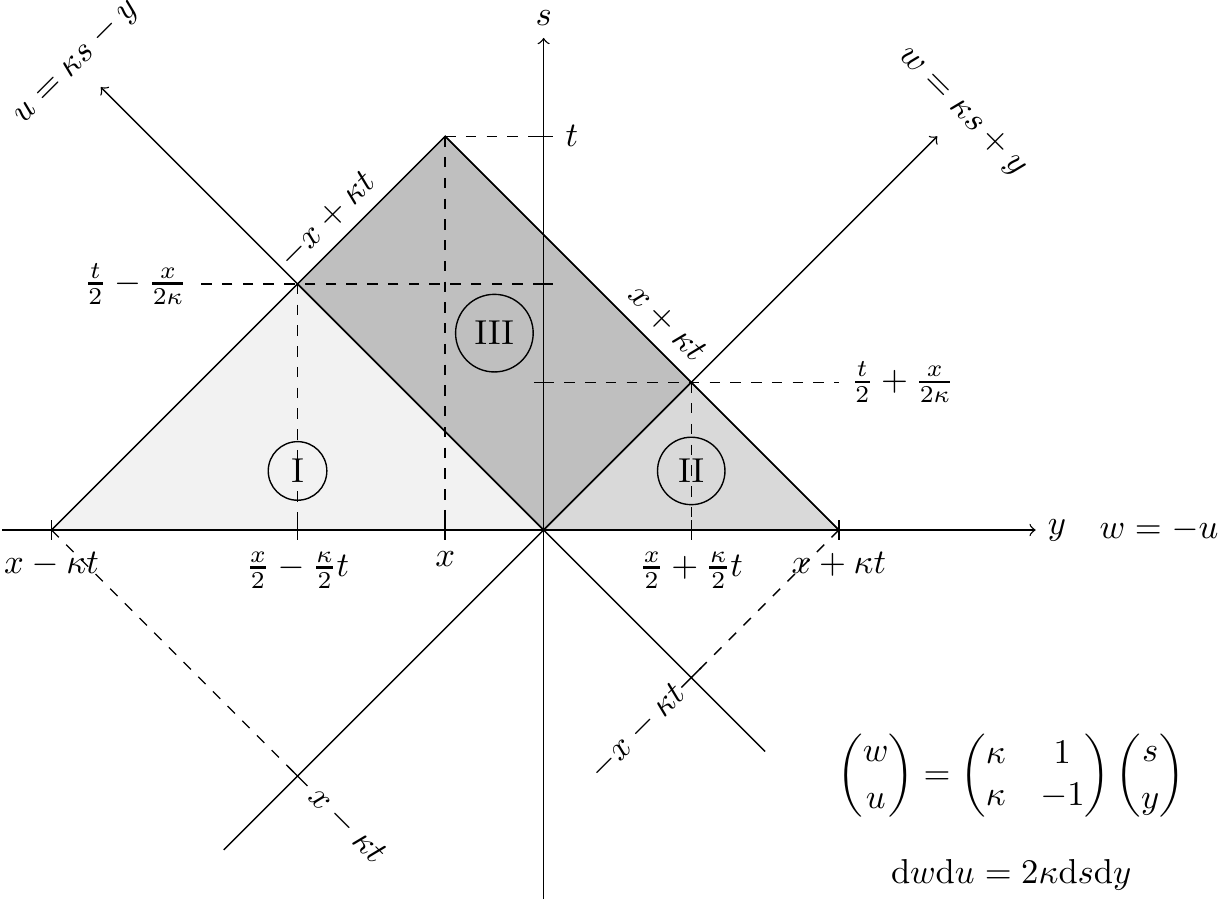}
\caption{Change variables for the case where  $|x|\le \kappa t$.}
\label{F4:ChangeVar}
\end{figure}

For all $n\in\bbN^*$ and $(t,x)\in\R_+^*\times\R$, recall that
$\calL_0(t,x;\lambda)= \lambda^2 G_\kappa^2(t,x)$
and $\calL_n(t,x;\lambda) = (\calL_0 \star \dots \star
\calL_0)(t,x)$, where there are $n+1$ convolutions
of $\calL_0(\cdot,\circ;\lambda)$.

\begin{proposition}\label{P4:K-W}
For all $n\in\bbN$, and $(t,x)\in\R_+^*\times\R$,
\begin{gather} \label{E4:Ln}
 \calL_n(t,x) =
\begin{cases}
  \frac{\lambda^{2n+2} \left((\kappa t)^2-x^2\right)^n}{2^{3n+2}
(n!)^2 \kappa^n} &\text{if $-\kappa t\le x\le \kappa t$,}\cr
0&\text{otherwise},
\end{cases}\\
\label{E4:K-Sum-Kn}
 \calK(t,x) = \sum_{n=0}^{\infty} \calL_n(t,x),\;\text{and}
\\
\label{E4:K-K0}
 \left(\calK \star \calL_0\right) (t,x) = \calK(t,x) -\calL_0(t,x)\:.
\end{gather}
Moreover, there are non-negative functions $B_n(t)$ such that for all
$n\in\bbN$, the function $B_n(t)$ is nondecreasing in $t$ and $\calL_n\le
\calL_0(t,x) B_n(t)$ for all $(t,x)\in\R_+^*\times\R$, and
\[\sum_{n=1}^\infty \left(B_n(t)\right)^{1/m}<+\infty,\quad
\text{for all $m\in\bbN^*$.}
\]
\end{proposition}
\begin{proof}
Formula \eqref{E4:Ln} clearly holds for $n=0$. By induction, suppose that it
is true for $n$. Now we evaluate
$\calL_{n+1}(t,x)$ from the definition and a change of variables (see
Figure \ref{F4:ChangeVar}):
\begin{align*}
\calL_{n+1}(t,x)&=\left(\calL_0\star\calL_n\right)(t,x)=\frac{\lambda^{2n+4
} } {
2^{3n+4}(n!)^2\kappa^n}\: \frac{1}{2\kappa}
\int_0^{x-\kappa t}\ud u\: u^n \int_0^{x+\kappa t}\ud w\: w^n  \\
&=
\frac{\lambda^{2(n+1)+2} \left((\kappa t)^2-x^2\right)^{n+1}}{2^{3(n+1)+2}
((n+1)!)^2 \kappa^{n+1}}
\end{align*}
for $-\kappa t\le x\le \kappa t$, and $\calL_{n+1}(t,x)=0$ otherwise. This
proves \eqref{E4:Ln}.
The series in \eqref{E4:K-Sum-Kn} converges to the modified Bessel
function of order zero by \eqref{E4:In}. As a direct consequence, we have
\eqref{E4:K-K0}.
Take $B_n(t) = \frac{\lambda^{2n} (\kappa t)^{2n}}{2^{3n} (n!)^2
\kappa^n}$,
which is non-negative and nondecreasing in $t$.
Then clearly, $\calL_n(t,x)\le \calL_0(t,x) B_n(t)$. To show the convergence,
by the ratio test, for all $m\in\bbN^*$, we have that
\[\frac{\left(B_n(t)\right)^{1/m}}{\left(B_{n-1}(t)\right)^{1/m}}
=\left(\frac{\lambda \sqrt{\kappa}\: t}{2\sqrt{2}}\right)^{\frac{2}{m}}
\left(\frac{1}{n}
\right)^{\frac{2}{m}}
\rightarrow 0,\]
as $n\rightarrow\infty$. This completes the proof.
\end{proof}

\begin{lemma}\label{L:K-W}
The kernel function $\calK(t,x)$ defined in \eqref{E4:BesselI} is strictly
increasing in $t$ for $x\in\R$ fixed and decreasing in $|x|$ for $t>0$ fixed.
Moreover, for all $\left(s,y\right)\in [0,t]\times\R$,
we have that
\[\frac{\lambda^2}{2} G_\kappa\left(s,y\right)\le
\calK\left(s,y\right)
\le
\frac{\lambda^2}{2} I_0\left(|\lambda|\sqrt{\kappa/2}\: t\right)
G_\kappa\left(s,y\right).\]
\end{lemma}
\begin{proof}
The first part is true by \eqref{E4:In}.
As for the inequalities, the upper bound follows from the first part.
The lower bound is clear since $I_0(0)= 1$ by \eqref{E4:In}.
\end{proof}


\begin{lemma}\label{L:Theta}
Recall the definition of $T_k(t,x)$ in \eqref{E4:Ttx}.
For all $t\in\R_+$, and $x,y\in\R$,
\begin{gather}\label{E4:GG}
G_\kappa(t-s,x-z)G_\kappa\!\left(t-s,y-z\right)\! =\!
\frac{1}{2}
G_\kappa\left(\!T_\kappa\left(t,x-y\right)-s,\frac{x+y}{2}-z\!\right)\!,\\
\label{E4:GG-1}
\int_{\R} \ud z \: G_\kappa(t,x-z)G_\kappa(t,y-z) = \frac{\kappa}{2}\:
T_\kappa \left(t,x-y\right),\;\text{and}
\\
\label{E4:GG-2}
\iint_{\R_+\times\R}\ud s \ud z\: G_\kappa(t-s,x-z)G_\kappa\left(t-s,y-z\right)
 =
\frac{\kappa}{4} \: T^2_\kappa \left(t,x-y\right).
\end{gather}
\end{lemma}

\begin{proof}
Since $G_\kappa(t-s,x-y)=\frac{1}{2}\Indt{\Lambda(t,x)}\left(s,y\right)$,
\eqref{E4:GG}--\eqref{E4:GG-2} are clear from Figure \ref{F4:WaveKernel}.
\end{proof}

\begin{figure}[h!tbp]
\centering
\subfloat[the case where $|x-y|\ge 2\kappa t$]{
   \includegraphics[scale =0.8] {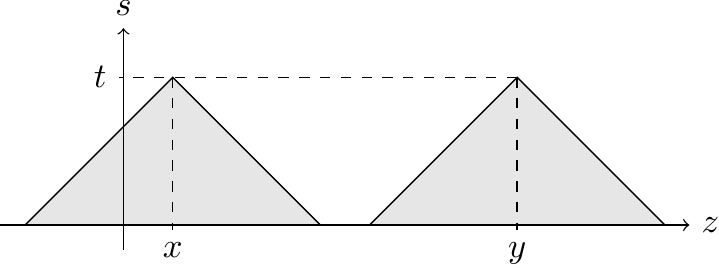}
 }
\qquad
 \subfloat[the case where $|x-y|< 2\kappa t$]{
   \includegraphics[scale =0.7] {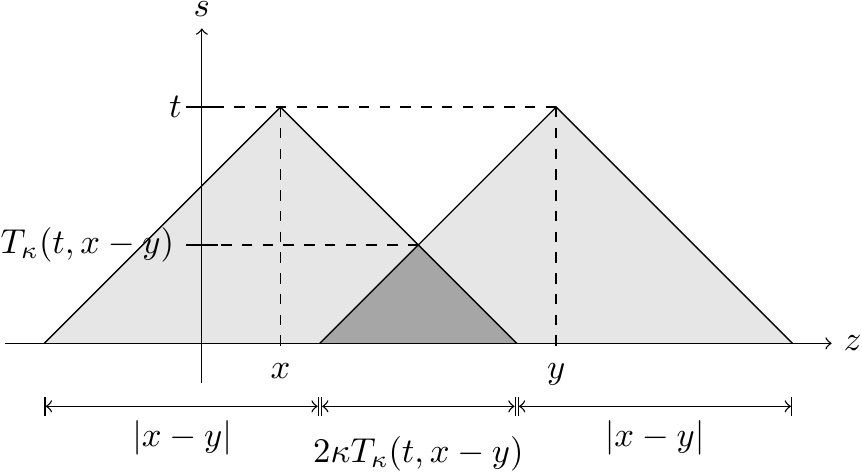}
 }
\caption{The two lightly shaded regions denote the support of the functions
$(s,z)\mapsto
G_\kappa(t-s,x-z)$ and $(s,z)\mapsto G_\kappa\left(t-s,y-z\right)$ respectively.
}
\label{F4:WaveKernel}
\end{figure}

\begin{proposition}\label{P4:G-Cont-Bd}
The wave kernel function $G_\kappa(t,x)$ satisfies Assumption
\ref{A:Cont-Bdd} with
$\tau=1/2$, $\alpha=\kappa/2$ and all $\beta\in \;]0,1[$ and $C=1$.
\end{proposition}
\begin{proof}
See Figure \ref{F4:G-W-Cont-Bd}.
The gray box is the
set $B_{t,x,\beta,\tau,\alpha}$. Clearly, we need $\alpha/\kappa +\tau
=1$. Therefore, we can choose $\tau=1/2$ and $\alpha=\kappa/2$.
\end{proof}

\begin{figure}[h!tbp]
\center
\includegraphics[scale=0.65]{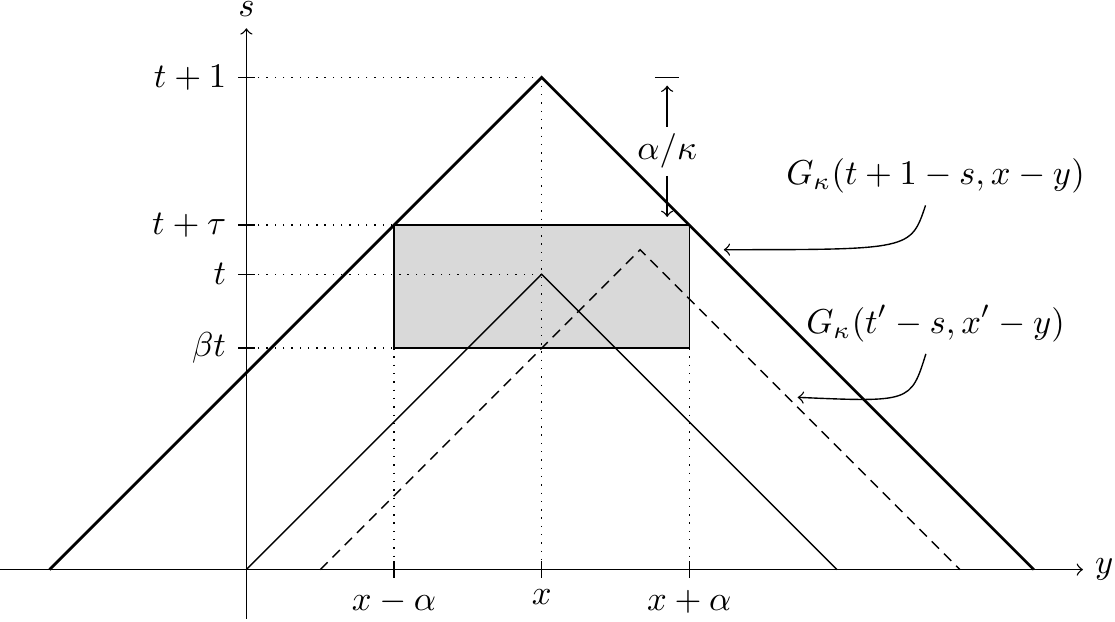}
\caption{$G_\kappa(t,x)$ verifies Assumption
\ref{A:Cont-Bdd}.}
\label{F4:G-W-Cont-Bd}
\end{figure}

For $g \in L_{loc}^{2}\left(\R\right)$ and $\mu\in\calM\left(\R\right)$,
define
\begin{gather}\label{E4:Psi-g}
\Psi_{g}(x) = \int_{-x}^x \ud y\: g^2(y),\quad \text{and}\quad
\Psi_{\mu}^*(x) = \left(|\mu|\left([-x,x]\right)\right)^2\;,\quad\text{for all
$x\ge 0$.}
\end{gather}
Clearly, these are nondecreasing functions of $x$.

\begin{lemma}\label{L:InDt}
If $g \in L_{loc}^2\left(\R\right)$ and $\mu\in \calM\left(\R\right)$, then
for all $v\in\R$ and
$(t,x)\in\R_+\times\R$,
\[\left(\left[v^2+J_0^2\right]\star
G_\kappa^2 \right) (t,x)\le
\frac{\kappa t^2}{4}\left(v^2+3\Psi_\mu^*\left(|x|+\kappa t\right)\right)  +
\frac{3}{16} \:t\:\Psi_g\left(|x|+\kappa
t\right)<+\infty.\]
Moreover, for all $v\in\R$ and all compact sets
$K\subseteq\R_+\times\R$,
\[
\sup_{(t,x)\in K}\left(\left[v^2+J_0^2\right]\star G_\kappa^2 \right) (t,x)
<+\infty.
\]
\end{lemma}

Note that the conclusion of this lemma is stronger than
Assumption \ref{A:InDt} since $t$ can be zero
here.

\begin{proof}
Suppose $t>0$.  Notice that $\left|(\mu * G_\kappa(s,\cdot))(y) \right|\le |\mu|
\left([y-\kappa s,y
+\kappa s]\right)$,
and so, recalling \eqref{E4:J0-Wave},
\begin{align*}
\left(\left[v^2+J_0^2\right]\star G_\kappa^2 \right)(t,x)
&= \frac{1}{4} \left(v^2\:
\iint_{\Lambda(t,x)}\ud s\ud y+\iint_{\Lambda(t,x)}\ud s\ud y\;
J_0^2\left(s,y\right)\right)\\
&\le \frac{1}{4}
\Bigg(v^2 \kappa t^2 + \frac{3}{4}
\int_0^t \ud s\int_{x-\kappa (t-s)}^{x+\kappa (t-s)}\ud y\;
\big(
g^2(y+\kappa s)+g^2(y-\kappa s) \\
& \hspace{40pt}+
4|\mu|^2\left([y-\kappa s, y+\kappa s]\right)
\big)\Bigg)\:.
\end{align*}
Clearly, for all $\left(s,y\right)\in \Lambda(t,x)$, by \eqref{E4:Psi-g},
\[
|\mu|^2\left([y-\kappa s, y+\kappa s]\right)
\le
|\mu|^2\left([x-\kappa t, x+\kappa t]\right)
\le \Psi_\mu^*\left(|x|+\kappa t\right)\:.
\]
The integral for $g^2$ can be easily evaluated by the change of variables in
Figure \ref{F4:ChangeVar}:
\begin{align*}
\int_0^t\ud s \int_{x-\kappa (t-s)}^{x+\kappa (t-s)} \left(
g^2(y+\kappa s)+g^2(y-\kappa s)
\right)\ud y
&= \frac{1}{2\kappa}\iint_{I \cup II \cup III}
\left(g^2(u)+g^2(w)\right) \ud u \ud w\\
&\le
\frac{1}{2\kappa}\int_{x-\kappa t}^{x+\kappa t}\ud w
\int_{-x-\kappa t}^{-x+\kappa t} \ud u
\left(g^2(u)+g^2(w)\right)\\
&\le t\: \Psi_g(|x|+\kappa t)\:,
\end{align*}
where $I$, $II$ and $III$ denote the three regions in Figure \ref{F4:ChangeVar}
and $\Psi_g$ is defined in \eqref{E4:Psi-g}.
Therefore,
\[
\left(\left[v^2+J_0^2\right]\star G_\kappa^2 \right)(t,x)\le
\frac{1}{4}\left(\left(v^2+3\Psi_\mu^*\left(|x|+\kappa
t\right)\right) \kappa t^2 + \frac{3}{4} \;t\;\Psi_g\left(|x|+\kappa
t\right)\right)<+\infty\:.
\]
Finally, let $a=\sup\big\{|x|+\kappa t: (t,x)\in
K\big\}$, which is finite because $K$ is a compact set. Then,
\[
\sup_{(t,x)\in K}\left(\left[v^2+J_0^2\right]\star G_\kappa^2 \right) (t,x)
\le
\frac{\kappa a^2}{4}\left(v^2+3\Psi_\mu^*\left(a\right)\right)  +
\frac{3}{16} \: a\: \Psi_g\left(a\right)
<+\infty\:,
\]
which completes the proof of Lemma \ref{L:InDt}.
\end{proof}

\begin{proof}[Proof of Theorem \ref{T4:ExUni}]
To apply Theorem \ref{T3:ExUni}, we need to verify the assumptions
\conRef{G} and \conRef{W} of Theorem \ref{T3:ExUni} with
$\theta(t,x)\equiv 1$.
We begin with \conRef{G}:
(a) is satisfied by \[\Theta_\kappa(t,x,x) =
\iint_{[0,t]\times\R} \ud s \ud y\: G_\kappa^2\left(t-s,x-y\right)=
\frac{\kappa t^2}{2}<+\infty\] and Proposition \ref{P4:K-W};
(b) is verified by Lemma \ref{L:InDt}.
\conRef{W} is true due to Proposition \ref{P4:G-Cont-Bd}.
As for the two-point correlation function, \eqref{E:TP} reduces to
\eqref{E4:TP} because, by \eqref{E4:GG},
\[\int_0^t \ud s\int_\R \ud z\: f(s,z)
G_\kappa(t-s,x-z)G_\kappa\left(t-s,y-z\right)
=
\frac{1}{2}\left(f\star
G_\kappa\right) \left(T_\kappa\left(t,x-y\right),\frac{x+y}{2}\right).
\]
This completes the proof of Theorem \ref{T4:ExUni}.
\end{proof}

\subsection{Weak intermittency}
\label{SS:Wave-WInterm}
Recall that $u(t,x)$ is said to be fully intermittent if
the Lyapunov exponent of order $1$ vanishes
and the lower Lyapunov exponent of order $2$ is strictly positive:
$m_1=0$ and $\underline{m}_2>0$.
The solution is called {\it weakly intermittent}
if  $\underline{m}_2>0$.

\begin{theorem}\label{T4:Intermit}
Suppose that $|\rho(u)|^2\le \Lip_\rho^2(\Vip^2+u^2)$, $g(x)\equiv w$ and
$\mu(\ud x) = \widetilde{w}\ud x$ with $w,
\widetilde{w}\in\R$.
Then we have the following two properties
\begin{enumerate}[(1)]
 \item For all even integers $p\ge 2$,
\begin{align}\label{E:mbds}
\overline{m}_p \le
\begin{cases}
\Lip_\rho \sqrt{2\kappa}\:  p^{3/2} & \text{if $\Vip\ne 0$ and $p>2$},\cr
\Lip_\rho  \sqrt{\kappa}\: p^{3/2} & \text{if $\Vip= 0$ and $p>2$},\cr
\Lip_\rho\sqrt{\kappa/2} & \text{if $p=2$}.
\end{cases}
\end{align}
\item If $|\rho(u)|^2\ge \lip_\rho^2(\vip^2+u^2)$ for some
$\lip_\rho \ne 0$, and if
$|\vip|+
|w|+|\widetilde{w}|\ne 0$ with $w \widetilde{w}\ge 0$, then
$\underline{m}_2\ge |\lip_\rho| \sqrt{\kappa/2}$
and so $u(t,x)$ is weakly intermittent.
\item If $|\rho(u)|^2=\lambda^2(\vv^2+u^2)$, with $\lambda\ne 0$, and if
$|\vip|+|w|+|\widetilde{w}|\ne 0$, then
$\underline{m}_2=\overline{m}_2=|\lambda|\sqrt{\kappa/2}$.
\end{enumerate}
\end{theorem}

\begin{proof}
Clearly, $J_0(t,x) = w + \kappa \widetilde{w} t$.
(1) If $|\Vip|+ |w|+|\widetilde{w}|= 0$, then $J_0(t,x)\equiv 0$ and
$\rho(0)=0$, so
$u(t,x)\equiv 0$ and the bound
is trivially true. If $|\Vip|+ |w|+|\widetilde{w}|\ne 0$, then by
\eqref{E4:SecMom-Up}, for all even integers $p\ge 2$,
\[\Norm{u(t,x)}_p^2\le 2 \left(w + \kappa \widetilde{w} t\right)^2 +
\left[\Vip^2
+2 \left(w +
\kappa \widetilde{w} t\right)^2\right]
\widehat{\calH}_p(t).
\]
Hence, by \eqref{E4:Sinh}, $\overline{m}_p \le a_{p,\Vip} z_p \Lip_\rho
\sqrt{\kappa/2}
\:p/ 2$. Then by \eqref{E:a_pv} and the fact that $z_2=1$ and $z_p\le 2\sqrt{p}$
for $p\ge 2$, we obtain \eqref{E:mbds}.

(2) Note that the term $ \Vip^2+2 \left(w+ \kappa \widetilde{w}t\right)^2$ on
the r.h.s. of the above inequality does not vanish since $|\Vip| +|w|
+|\widetilde{w}|\ne 0$. By \eqref{E4:SecMom-Lower} and Corollary
\ref{C4:HomeInit},
\[\Norm{u(t,x)}_2^2 \ge
-\vip^2 -\frac{4\kappa \widetilde{w}^2}{\lip_\rho^2} +
\left(w^2 + \vip^2 + \frac{4\kappa
\widetilde{w}^2}{\lip_\rho^2}\right)\cosh\left(|\lip_\rho|\sqrt{\kappa
/2}\: t\right).\]
Clearly, $|\vip| +|w|
+|\widetilde{w}|\ne 0$ implies that $\underline{m}_2 \ge  |\lip_\rho|
\sqrt{\kappa/2}$.

Part (3) is a consequence of (1) and (2). This completes the proof of
Theorem
\ref{T4:Intermit}.
\end{proof}

\begin{remark}
It would be interesting to obtain a lower bound of the form $\underline{m}_p\ge
C p^{3/2}$.
Dalang and Mueller \cite{DalangMueller09Intermittency} derived
the lower bound for the stochastic wave and heat equations in $\R_+\times\R^3$
in the case where $\rho(u) = \lambda u$ and the driving noise is spatially
colored. An essential tool in their paper is a Feynman-Kac-type formula that
they obtained (with Tribe) in \cite{DalangMT06FKT}.
\end{remark}

\subsection{Exponential growth indices}
\label{SS:Wave-Exp}
Recall the definition of $\underline{\lambda}_p(x)$ and
$\overline{\lambda}_p(x)$ in \eqref{E4:GrowInd-0} and \eqref{E4:GrowInd-1}.
Define
\begin{align}\label{E2:MGa}
\calM_G^{\sd}\left(\R\right):= \left\{\mu\in\calM\left(\R\right): \: \int_\R
e^{\sd|x|}|\mu|
(\ud x)<+\infty\right\}\;, \quad \sd\ge 0.
\end{align}
We use subscript ``$+$'' to denote the subset of non-negative measures.
For example, $\calM_+\left(\R\right)$ is the set of non-negative Borel measures
over $\R$
and
$\calM_{G,+}^{\sd}\left(\R\right) = \calM_G^{\sd}\left(\R\right) \cap
\calM_+\left(\R\right)$.

\begin{remark}\label{R:WaveGI}
Since the kernel function $\calK(t,x)$ has support in the same space-time cone
as the fundamental solution $G_\kappa(t,x)$, it is clear that if the initial
data have compact support, then the solution, including any high peaks related
to intermittency, must propagate in the space-time cone with the same speed
$\kappa$. Hence
$\underline{\lambda}(p)\le \overline{\lambda}(p) \le \kappa$. Conus and
Khoshnevisan showed in \cite[Theorem 5.1]{ConusKhosh10Farthest} that with some
other mild conditions on the compactly supported initial data,
$\underline{\lambda}(p)= \overline{\lambda}(p) = \kappa$ for all $p\ge 2$.
\end{remark}

\begin{theorem}\label{Tw:Growth}
We have the following:
\begin{enumerate}[(1)]
 \item Suppose that $|\rho(u)|\le \Lip_\rho |u|$ with $\Lip_\rho\ne 0$ and the
initial data satisfy the following two conditions:
\begin{enumerate}[(a)]
 \item The initial position $g(x)$ is a Borel function such
that $|g(x)|$ is bounded from above by some function $c e^{-\sd_1 |x|}$ with
$c>0$ and $\sd_1>0$ for almost all $x\in\R$;
 \item The initial velocity $\mu\in \calM_G^{\sd_2}\left(\R\right)$ for some
$\sd_2>0$.
\end{enumerate}
Then for all even integers
$p\ge 2$,
\[\overline{\lambda}(p)\le
\begin{cases}
\displaystyle
\kappa \left(1+\frac{a_{p,\Vip}^2 \: z_p^2 \: \Lip_\rho^2
}{8\kappa\: (\sd_1\wedge \sd_2)^2}\right)^{1/2}
&
\text{if $p>2$},\\[1em]
\displaystyle
\kappa \left(1+\frac{\Lip_\rho^2
}{8\kappa\: (\sd_1\wedge \sd_2)^2}\right)^{1/2}
&
\text{if $p=2$}.
\end{cases}\]
\item Suppose that $|\rho(u)|\ge \lip_\rho |u|$ with $\lip_\rho\ne 0$ and the
initial data satisfy one of the following two conditions:
\begin{enumerate}[(a')]
 \item The initial position $g(x)$ is a non-negative Borel function
bounded from below by some function $c_1 e^{-\sd_1' |x|}$ with $c_1>0$ and
$\sd_1' >0$ for almost all $x\in\R$;
 \item The initial velocity $\mu(\ud x)$ has a density $\mu(x)$ that is a
non-negative Borel function bounded from below by some function $c_2
e^{-\sd_2' |x|}$ with $c_2>0$ and $\sd_2'>0$ for almost all $x\in\R$.
\end{enumerate}
Then
\[\underline{\lambda}(p) \ge
\kappa \left(1+\frac{\lip_\rho^2}{8\kappa \left(\sd_1'\wedge
\sd_2'\right)^2}\right)^{1/2},\quad\text{for all even integers $p\ge 2$.}
\]
\end{enumerate}
In particular, we have the following two special cases:
\begin{enumerate}
 \item[(3)] For the hyperbolic Anderson model $\rho(u) = \lambda u$ with
$\lambda\ne 0$, if the initial velocity
$\mu$ satisfies all Conditions (a), (b), (a') and (b') with
$\sd:=\sd_1\wedge \sd_2 = \sd_1'\wedge \sd_2'$,
then
\[\underline{\lambda}(2)   =
\overline{\lambda}(2) =\kappa\left(1+\frac{\lambda^2}{8\kappa
\sd^2}\right)^{1/2}.\]
\item[(4)] If $\lip_\rho |u|\le |\rho(u)|\le \Lip_\rho |u|$ with $\lip_\rho\ne
0$ and $\Lip_\rho \ne 0$, and both $g(x)$ and $\mu(x)$ are non-negative Borel
functions with compact support, then
\[
\overline{\lambda}(p) = \underline{\lambda}(p) =\kappa,\quad
\text{for all even integers $p\ge 2$.}
\]
\end{enumerate}
\end{theorem}
\begin{proof}
The statements of (1) and (2) are a consequence of Propositions \ref{P4:InitPos}
and \ref{P4:InitVel} below.
More precisely, let $J_{0,1}(t,x)$ (resp. $J_{0,2}(t,x)$) be the homogeneous
solutions obtained with the initial data $g$ and $0$ (resp. $0$ and
$\mu$). Clearly, $J_0(t,x) = J_{0,1}(t,x)+ J_{0,2}(t,x)$.
For the upper bounds, we use the fact that $J_0^2(t,x) \le 2 J_{0,1}^2(t,x)+
2J_{0,2}^2(t,x)$.
By \eqref{E4:SecMom-Up}, we simply choose the larger of the upper bounds between
Proposition \ref{P4:InitPos} (1) and Proposition \ref{P4:InitVel} (1).
As for the lower bounds, because
both $g$ and $\mu$ are nonnegative, $J_0^2(t,x) \ge J_{0,1}^2(t,x)+
J_{0,2}^2(t,x)$.
Hence, by \eqref{E4:SecMom-Lower},
we only need to take the larger of the lower bounds between Proposition
\ref{P4:InitPos} (2) and Proposition \ref{P4:InitVel} (2).
Part (3) is a direct consequence of (1) and (2).
When the initial data have compact support, both (1) and (2) hold for all
$\sd_i>0$ with $i=1,2$. Then
letting these $\sd_i$'s tend to $+\infty$ proves (4).
\end{proof}

Note that for Conclusion (3), clearly, $\sd_i'\ge \sd_i$, $i=1,2$. Hence,
the condition $\sd_1\wedge\sd_2=\sd_1'\wedge\sd_2'$ has only two possible
cases: $\sd_1'=\sd_1 \le \sd_2\le \sd_2'$
 and $\sd_2'=\sd_2 \le \sd_1\le\sd_1'$.

\begin{remark}
The behaviour of growth indices of the solution to the
stochastic wave equation \eqref{E4:WaveInt} depends on the growth rate of the
nonlinearity of $\rho$, and also on the rate of decay at $\pm\infty$ of the
initial data. In particular, the initial data significantly affects the
behavior of the solution for all time. However, when the initial data are
compactly supported,
the growth rate of the non-linearity  $\rho$ plays no role.
\end{remark}

\subsection{Two propositions for the exponential growth indices}
\label{SS:Wave-GrowInd}
The following asymptotic
formula for $I_0(x)$
(see, \cite[(10.30.4)]{NIST2010}) will be useful
\begin{align}\label{E4:AsyIn}
I_0(x) \sim \frac{e^x}{\sqrt{2\pi x}},\quad
\text{as $x\rightarrow \infty$.}
\end{align}

\subsubsection{Contributions of the initial position}
First consider the case where $\mu \equiv 0$.
Recall that $H(t)$ is the Heaviside function.

\begin{lemma}\label{L:InitPos-Bd}
Let $f(t,x)=\frac{1}{2}\left(e^{-\sd|x-\kappa t|}+e^{-\sd|x+\kappa
t|}\right) H(t)$. Then we have the following bounds:
\begin{enumerate}[(1)]
 \item Set $\sigma:= \sqrt{  \sd ^2
+\frac{\lambda^2}{2\kappa} }$.
For $\sd>0$, $t\ge 0$ and $|x|\ge \kappa t$,
\[
\left(f\star \calK \right)(t,x) \le
 \frac{\lambda^2 \: t}{2(\sigma-\sd)} e^{-\sd|x|+\kappa\sigma t} \:.
\]
\item For $(t,x)\in \R_+^*\times\R$, $\sd>0$ and $a,b \in \;]0,1[$,
\[
\left(f\star \calK\right)(t,x)\ge
\begin{cases}
  \frac{1}{2}e^{-\sd \kappa t}\cosh(\sd |x|)
\left(I_0\left(\sqrt{\frac{\lambda^2(\kappa^2 t^2-x^2 )}{2\kappa}}\right)-1
\right)& \text{if $|x|\le \kappa t$,}\cr
\frac{\lambda^2 e^{-\sd|x|}}{2(1-a^2)\sd^2 \kappa}
I_0\left(\sqrt{\frac{\lambda^2(1-a^2 )}{2\kappa}}\: b \:\kappa t\right)
g(t\;;a,b,\sd,\kappa) & \text{if $|x|\ge \kappa t$\:,}
\end{cases}
\]
where the function $g\left(t\;;a,b,\sd,\kappa\right)$ is equal to
\[a\cosh\left(a b \sd \kappa t \right)\cosh\left((1-b)\sd \kappa t\right)
 -a\cosh\left(a \sd \kappa t\right)
 + \sinh\left((1-b)\sd \kappa t\right) \sinh\left(a b \sd
\kappa t\right).\]
\end{enumerate}
\end{lemma}

\begin{proof}
(1) Because $f(t,\circ)$ and $\calK(t,\circ)$ are even functions, it suffices
to consider the case $x\le -\kappa t$. In this case,
$y\le -\kappa s$ implies that
$f(s,y)= \frac{1}{2}\left(e^{\sd(y-\kappa s)}+e^{\sd(y+\kappa s)}\right) H(s)$.
Hence, by \eqref{E4:I0Exp},
\begin{align*}
\left(f\star \calK\right)(t,x) &\le \frac{\lambda^2}{4}\int_0^t\ud s
\int_{x-\kappa (t-s)}^{x+\kappa (t-s)}\ud y\;
\frac{1}{2}\left(e^{\sd(y-\kappa s)}+e^{\sd(y+\kappa s)}\right)
\exp\left(\sqrt{\frac{\lambda^2[\kappa^2 (t-s)^2-(x-y)^2]}{2\kappa}}\right)\\
&=
 \frac{\lambda^2}{8}\int_0^t\ud s\: \left(e^{\sd(x-\kappa
(t-s))}+e^{\sd(x+\kappa (t-s))}\right)
\int_{-\kappa s}^{\kappa s}\ud y\;
\exp\left(-\sd y+ \sqrt{\frac{\lambda^2[\kappa^2 s^2-y^2]}{2\kappa}}\right).
\end{align*}
The function $\psi(y):= -\sd y+ \left[\lambda^2(\kappa^2
s^2-y^2)/(2\kappa)\right]^{1/2}$ achieves its
maximum at $y=-\sigma^{-1}\sd\kappa s \in [-\kappa s, \kappa s]$, and
$\max_{|y|\le \kappa s} \psi(y) = \sigma \kappa s$, so
\begin{align*}
\left(f\star \calK\right)(t,x) &\le
\frac{\lambda^2 \kappa \: t}{4}\int_0^t\ud s\:
 \left(e^{\sd(x-\kappa t) + \kappa (\sigma+\sd) s}+e^{\sd(x+\kappa t) + \kappa
(\sigma-\sd ) s}\right)\\
& \le \frac{\lambda^2 t}{4 (\sigma -\sd)}
\left(e^{\sd(x-\kappa t)+ \kappa (\sigma +\sd) t}+e^{\sd(x+\kappa t)
+ \kappa (\sigma -\sd)t}\right)
= \frac{\lambda^2 t}{2(\sigma -\sd)} e^{\sd x + \kappa \sigma t}\;.
\end{align*}

(2) We consider two cases. {\em Case I:} $|x|\le \kappa t$.
As shown in Figure \ref{F4:ChangeVar}, we decompose the space-time convolution
into three parts $S_i$ corresponding to the three integration regions $D_i$,
$i=1,2,3$:
\[
\left(f\star G_\kappa\right)(t,x) = \sum_{i=1}^3 S_i = \sum_{i=1}^3
\frac{1}{2}\iint_{D_i} \ud s\ud y\: f(s,y).
\]
Clearly, $\left(f \star \calK \right)(t,x)  \ge S_3$.
Because
\[f\left(s,y\right)\ge
\frac{1}{2}\left(e^{-\beta(\kappa
t-x)}+e^{-\beta(\kappa t+x)}\right),
\quad \text{for all $(s,y)\in D_3$},
\]
we see that
\[S_3\ge \frac{2}{\lambda^2}
e^{-\beta\kappa
t}\cosh\left(\beta x\right)
\left(\calL_0 \star \calK\right)(t,x).
\]
Then apply \eqref{E4:K-K0}.

{\em Case II:} $|x| \ge \kappa t$. Similar to the proof of part (1), one can
assume
that $x\le -\kappa t$. Then
\begin{align*}\notag
\left(f\star\calK\right)(t,x)= \frac{\lambda^2}{8}\int_0^t \ud s\int_{-\kappa
s}^{\kappa s}  \ud y\:
I_0\left(\sqrt{\frac{\lambda^2(\kappa^2s^2-y^2)}{2\kappa}}\right)
\left(e^{\sd(x-y-\kappa(t-s))}+e^{\sd(x-y+\kappa(t-s))}\right).
\end{align*}
Fix $a,b\in \;]0,1[\:$. Then
\begin{align*}
\left(f\star\calK\right)(t,x)\ge&\;
\frac{\lambda^2}{4}\int_{b t}^t \ud s \int_{-a \kappa
s}^{a\kappa s} \ud y\:
I_0\left(\sqrt{\frac{\lambda^2(\kappa^2s^2-y^2)}{2\kappa}}\right)
 e^{\sd (x-y)}\cosh(\sd \kappa(t-s)) \\ \notag
\ge&\;
\frac{\lambda^2 e^{\sd x}}{4}
I_0\left(\sqrt{\frac{\lambda^2(1-a^2)}{2\kappa}}\:b\: \kappa t\right)
\int_{b t}^t
\ud s\int_{-a \kappa s}^{a\kappa s}\ud y\;
\cosh(\sd\kappa(t-s)) e^{-\sd y}\:.
\end{align*}
Since
\[\int_{b t}^t
\ud s\int_{-a \kappa s}^{a\kappa s}\ud y\:
\cosh(\sd \kappa(t-s)) e^{-\sd y}
=\frac{2}{\sd} \int_{bt}^t \ud s\: \cosh(\sd\kappa(t-s))
\sinh(a \sd \kappa s),\]
part (2) is proved by an application of the integral in Lemma \ref{L:coshsinh}.
\end{proof}

\begin{proposition}\label{P4:InitPos}
Suppose that $\mu \equiv 0$. Fix $\sd>0$. Then:
\begin{enumerate}[(1)]
 \item Suppose $|\rho(u)|\le \Lip_\rho |u|$ with $\Lip_\rho\ne 0$ and let $g(x)$
be a measurable function such that
for some constant $C>0$, $|g(x)|\le C e^{-\sd |x|}$ for almost
all $x\in\R$. Then
\begin{align}\label{E:WGIU}
\overline{\lambda}(p) \le
\begin{cases}\displaystyle
\kappa \left(1+\frac{a_{p,\Vip}^2\: z_p^2 \: \Lip_\rho^2}{8\kappa
\sd^2}\right)^{1/2}
& \text{if $p>2$ is an even integer},
\\[1em]
\displaystyle
\kappa \left(1+\frac{\Lip_\rho^2}{8\kappa
\sd^2}\right)^{1/2}
&\text{if $p=2$}\:.
\end{cases}
\end{align}
\item Suppose $|\rho(u)|\ge \lip_\rho |u|$ with $\lip_\rho\ne 0$ and let $g(x)$
be a measurable function such
that for some constant $c>0$, $|g(x)| \ge c\: e^{-\sd|x|}$ for
almost all $x\in\R$. Then
\begin{align}\label{E:WGIL}
\underline{\lambda}(p) \ge \kappa \left(1+\frac{\lip_\rho^2}{8\kappa
\sd^2}\right)^{1/2},
\quad\text{for all even integers $p\ge 2$.}
\end{align}
\end{enumerate}
In particular, if $g(x)$ satisfies both Conditions (1) and (2), and $\rho(u) =
\lambda u$ with
$\lambda \ne 0$, then
\begin{align}\label{E:WGIE}
\underline{\lambda}(2) = \overline{\lambda}(2) =
\kappa \left(1+\frac{\lambda^2}{8\kappa
\sd^2}\right)^{1/2}.
\end{align}
\end{proposition}

\begin{proof}
(1) Let $J_0(t,x)= \frac{1}{2}\left(g(x-\kappa t)+g(x+\kappa t)\right) H(t)$.
By the assumptions on $g(x)$,
\[\left|J_0(t,x)\right|^2\le \frac{C^2}{2}
\left(e^{-2\sd|x-\kappa
t|}+e^{-2\sd|x+\kappa t|}\right) H(t),\quad\text{for almost all
$(t,x)\in\R_+\times\R$.}\]
We first consider the case $p>2$.
By the moment formula \eqref{E4:SecMom-Up} and Lemma \ref{L:InitPos-Bd} (1),
for $|x|\ge \kappa t$,
\[
\Norm{u(t,x)}_p^2 \le
2J_0^2(t,x)+ C' t \exp\left(-2 \sd|x|+ \kappa \sigma t \right),
\]
for some constant $C'>0$, where $\sigma := \left[4\sd^2 +
(2\kappa)^{-1}a_{p,\Vip}^2 \:z_p^2 \Lip_\rho^2\:\right]^{1/2}$.
We only need to consider the case where $\alpha>\kappa$; see Remark
\ref{R:WaveGI}.
Because the supremum over $|x|\ge \alpha t$ of the right-hand side is attained
at $|x|=\alpha t$,
\[
\lim_{t\rightarrow\infty}\frac{1}{t} \sup_{|x|\ge \alpha t} \log
\Norm{u(t,x)}_p^p
\le - 2\alpha \sd+ \kappa \sigma ,\quad \text{for
$\alpha> \kappa$.}
\]
Solve the inequality $- 2 \alpha\sd + \kappa \sigma<0$ to get
$\overline{\lambda}(p) \le \kappa \frac{\sigma}{2\sd}$,
which is the formula in \eqref{E:WGIU} for $p>2$.
For the case $p=2$, we
simply replace $z_p$ and $a_{p,\Vip}$ by $1$ (see \eqref{E:a_pv}).

(2) Note that $\underline{\lambda}(p)\ge \underline{\lambda}(2)$, because
$\Norm{u}_p\ge \Norm{u}_2$ for $p\ge 2$,
we only need to consider $p=2$. Assume first that
$\rho(u)=\lambda u$.  Since $|g(x)|\ge c\: e^{-\sd|x|}$ a.e.,
\[
J_0^2(t,x) \ge
\frac{c^2}{4} \left( e^{-2\sd|x-\kappa t|}+e^{-2\sd|x+\kappa
t|}\right).
\]

If $|x|\le \kappa t$, by \eqref{E4:SecMom-Lower}, Lemma \ref{L:K-W} and Lemma
\ref{L:InitPos-Bd},
\[\Norm{u(t,x)}_2^2\ge
\left(J_0^2\star\calK\right)(t,x) \ge
\frac{c^2}{4} e^{-2\sd \kappa t}\cosh(2\sd |x|)
\left(I_0\left(\sqrt{\frac{\lambda^2(\kappa^2t^2-x^2)}{2\kappa}}
\right)-1\right).
\]
Hence, for $0\le \alpha < \kappa$, by \eqref{E4:AsyIn},
\[
\lim_{t\rightarrow+\infty} \frac{1}{t}\sup_{|x|\ge \alpha t} \log
\Norm{u(t,x)}_2^2
\ge -2\sd \kappa + 2\sd \alpha+
|\lambda|\sqrt{\frac{\kappa^2-\alpha^2}{2\kappa}}\:.
\]
Then
\[h(\alpha):= -2\sd \kappa + 2\sd \alpha+
\frac{|\lambda|}{\sqrt{2\kappa}}\sqrt{\kappa^2-\alpha^2} \ge 0
\quad\Leftrightarrow\quad
\kappa\;\frac{8\kappa\sd^2
-\lambda^2}{8\kappa\sd^2 +\lambda^2} \le \alpha  \le \kappa.
\]
As $\alpha$ tends to $\kappa$ from the left side, $h(\alpha)$ remains positive.
Therefore, $\underline{\lambda}(2) \ge \kappa$.

If $x\le -\kappa t$, again, by Lemma \ref{L:InitPos-Bd},
\[\Norm{u(t,x)}_2^2
\ge
\frac{c^2 \lambda^2 e^{-2\sd|x|}}{4(1-a^2)(2\sd)^2 \kappa}
I_0\left(\sqrt{\frac{\lambda^2(1-a^2)}{2\kappa}}b\kappa t\right)
g(t\;;a,b,2\sd,\kappa), \quad\text{for all $a,b\in \;]0,1[$.}
\]
For large $t$, replace both $\cosh(C t) $ and $\sinh(C t)$ by
$\exp(C t)/2$, with $C\ge 0$, to see that
\[g(t \;;a,b,2\sd,\kappa) \ge C'
\exp\left(2(1+(a-1)b) t\sd \kappa \right),\]
for some constant $C'>0$.
Hence, for $\alpha > \kappa$, by \eqref{E4:AsyIn},
\[
\lim_{t\rightarrow\infty}\frac{1}{t} \sup_{|x|\ge \alpha t} \log
\Norm{u(t,x)}_2^2
\ge
\sqrt{\frac{\lambda^2(1-a^2)}{2\kappa}} b \kappa  - 2\sd \alpha +
2(1-(1-a)b)\sd \kappa\;.
\]
Solve the inequality
\[h(\alpha):= \sqrt{\frac{\lambda^2(1-a^2)}{2\kappa}} b \kappa  -
2 \sd \alpha +
2 (1-(1-a)b)\sd
\kappa>0\]
to get
\[\alpha < \left(\sqrt{\frac{\lambda^2(1-a^2)}{2\kappa}}
\frac{b}{2\sd}
+1-(1-a)b\right) \kappa.\]
Since $a\in \;]0,1[$ is arbitrary, we can choose
\[a:= \mathop{\arg\max}_{a\in\;]0,1[}
\left(\sqrt{\frac{\lambda^2(1-a^2)}{2\kappa}} \frac{b}{2\sd}
+1-(1-a)b\right)
= \left(1+\frac{\lambda^2}{8\kappa \sd^2}\right)^{-1/2}.\]
In this case, the critical growth rate is
$\alpha = b \kappa \left[1+\lambda^2/(8\kappa
\sd^2)\right]^{1/2}+(1-b)\kappa$.
Finally, since $b$ can be arbitrarily close to $1$, we have that
$\underline{\lambda}(2)\ge \kappa \left[1+\lambda^2/(8\kappa
\sd^2)\right]^{1/2}$,
and for the general case $|\rho(u)|\ge \lip_\rho|u|$, we have that
$\underline{\lambda}(p)\ge \underline{\lambda}(2)\ge \kappa
\left[1+\lip_\rho^2/(8\kappa \sd^2)\right]^{1/2}$.
This completes the proof of Proposition \ref{P4:InitPos}.
\end{proof}

\subsubsection{Contributions of the initial velocity}
Now, let us consider the case where $g(x)\equiv 0$.
We shall first study the case where $\mu(\ud x)=e^{-\sd |x|}\ud x$ with $\sd>0$.
In this case, $J_0(t,x)$ is given by the
following lemma.

\begin{lemma} \label{L:InitVel-Ind}
Suppose that $\mu (\ud x)= e^{-\sd |x|}\ud x$ with $\sd>0$.
For all $(t,x)\in \R_+\times\R$ and $z>0$,
\[\left(\mu * \Indt{|\cdot|\le z}\right)(x)=
\begin{cases}
2 \sd^{-1} e^{-\sd |x|} \sinh(\sd z) & |x|\ge z,\cr
2 \sd^{-1} \left(1-e^{-\sd z} \cosh(\sd x)\right) & |x|\le z.
\end{cases}\]
In particular, we have that
$
J_0(t,x) =
\begin{cases}
\sd^{-1} e^{-\sd |x|} \sinh(\sd \kappa t) & |x|\ge \kappa t,
\cr
\sd^{-1} \left(1-e^{-\sd\kappa t} \cosh(\sd x)\right) & |x|\le \kappa
t.
\end{cases}
$
\end{lemma}
The proof is straightforward, and is left to the
reader (see also \cite[Lemma 4.4.5]{LeChen13Thesis}).

\begin{lemma}\label{L:InitVel-Bd}
Suppose that $\mu\in\calM_G^{\sd}\left(\R\right)$ with $\sd>0$.  Set
$h(t,x)=\left(\mu * G_\kappa(t,\cdot)\right)(x)$ and
$\sigma=\left[\sd^2+(2\kappa)^{-1}\lambda^2\right]^{1/2}$.
Then for all $t\ge 0$ and $x\in\R$,
\begin{gather*}
\left| h(t,x)\right| \le C \exp\left(\sd
\kappa t -\sd |x|\right),\quad\text{with $C= 1/2\int_\R |\mu|(\ud
x)\: e^{\sd |x|}$}\:,
\end{gather*}
and
\begin{gather*}
\left( |h| \star \calK\right)(t,x)
\le\frac{\lambda^2 t}{2 (\sigma-\sd )} e^{-\sd |x|+ \sigma \kappa t}.
\end{gather*}
\end{lemma}

\begin{proof}
Considering the first inequality, observe that
\begin{align*}
e^{\sd |x|}\left|\left(\mu * G_\kappa(t,\cdot)\right)(x)\right|
&\le \frac{1}{2} \int_{x-\kappa t}^{x+\kappa t}|\mu|(\ud y)\: e^{\sd |x|}
\le
\frac{1}{2} \int_{x-\kappa t}^{x+\kappa t} |\mu|(\ud y) \:e^{\sd |x-y|} e^{\sd
|y|}\\
&\le \frac{1}{2} e^{\sd \kappa t}
\int_{x-\kappa t}^{x+\kappa t}|\mu|(\ud y)\: e^{\sd |y|}
\le \frac{1}{2} e^{\sd \kappa t} \int_\R |\mu|(\ud y)\:
e^{\sd |y|}.
\end{align*}
For the second inequality, set $f(t,x) = e^{\sd \kappa t-\sd
|x|}$. Then by \eqref{E4:I0Exp},
\begin{align*}
\left(f\star \calK\right)(t,x)
& = \frac{\lambda^2}{4}\int_0^t\ud s\: e^{\sd \kappa(t-s)}
\int_{-\kappa s}^{\kappa s}\ud y \:
\exp\left(
-\sd |x-y| +\sqrt{\frac{\lambda^2\left(\kappa^2 s^2 -y^2\right)}{2\kappa}}
\right)\\
&\le
\frac{\lambda^2}{4}\int_0^t\ud s\: e^{\sd \kappa(t-s)}
\int_{-\kappa s}^{\kappa s}\ud y \:
\exp\left(
-\sd |x|+ \sd |y| +\sqrt{\frac{\lambda^2\left(\kappa^2 s^2
-y^2\right)}{2\kappa}}
\right) \\
&\le
\frac{\lambda^2}{2}e^{-\sd |x|}\int_0^t\ud s\: e^{\sd \kappa(t-s)}
\int_{0}^{\kappa s}\ud y \:
\exp\left( \sd y +\sqrt{\frac{\lambda^2\left(\kappa^2 s^2
-y^2\right)}{2\kappa}}
\right) .
\end{align*}
The function $\psi(y):= \sd y +\left[\lambda^2\left(\kappa^2 s^2
-y^2\right)/(2\kappa)\right]^{1/2}$ achieves its
maximum at $y=\sigma^{-1}\sd \kappa s\in
[0,\kappa s]$,  and $\max_{y\in [0,\kappa s]} \psi(y) = \sigma \kappa s$,
so
\begin{align*}
 \left(f\star\calK\right)
& \le \frac{\lambda^2 \kappa t}{2}  e^{-\sd |x|}\int_0^t \ud s \: e^{\sd \kappa
(t-s) +
\sigma \kappa s}  \le
\frac{\lambda^2 t}{2 (\sigma-\sd )} e^{-\sd |x|+ \sigma \kappa t}\:.
\end{align*}
This completes the proof.
\end{proof}

\begin{proposition} \label{P4:InitVel}
 Suppose that  $g\equiv 0$. Fix $\sd>0$.
\begin{enumerate}[(1)]
 \item  If $|\rho(u)|\le \Lip_\rho |u|$ with $\Lip_\rho \ne 0$ and $\mu\in
\calM_G^{\sd}\left(\R\right)$, then $\overline{\lambda}(p)$ satisfies
\eqref{E:WGIU}.
\item Suppose that $|\rho(u)|\ge \lip_\rho |u|$ with $\lip_\rho\ne 0$ and
$\mu(\ud x)=f(x)\ud x$. If for some constant $c>0$,
$f(x) \ge c e^{-\sd |x|}$ for all almost all $x\in\R$,
then $\underline{\lambda}(p)$ satisfies \eqref{E:WGIL}.
\end{enumerate}
In particular, if  $\mu$ satisfies both Conditions (1) and (2), and $\rho(u) =
\lambda u$ with $\lambda\ne 0$, then
\eqref{E:WGIE} holds.
\end{proposition}
\begin{proof}
(1) Let $p> 2$ be an even integer. Let $h(t,x)$ be the function
defined in Lemma \ref{L:InitVel-Bd}.
Notice that the first bound in Lemma \ref{L:InitVel-Bd} is satisfied by
$h^2(t,x)$ provided $\sd$ is replaced by $2\sd$.
By \eqref{E4:SecMom-Up} and Lemma \ref{L:InitVel-Bd}, we see that for some
constant $C'>0$,
\[\Norm{u(t,x)}_p^2 \le
2 h^2(t,x) +C' t\exp\left(-2\sd|x|+\kappa \sigma t\right),
\]
where $\sigma=\left[4\sd^2+a_{p,\Vip}^2\: z_p^2 \:
\Lip_\rho^2/(2\kappa)\right]^{1/2}$. Then it is clear that
\begin{align*}
 \lim_{t\rightarrow\infty}\frac{1}{t} \sup_{|x|\ge \alpha t}
\log\Norm{u(t,x)}_p^p
\le -2\sd \alpha + \kappa \sigma.
\end{align*}
Solve the inequality $-2 \sd \alpha + \kappa \sigma>0$ to get
$\overline{\lambda}(p) \le \kappa \frac{\sigma}{2\sd}$.
For the case $p=2$, simply replace $z_p$ and $a_{p,\Vip}$ by $1$.

(2) Suppose that $f(x)\ge e^{-\sd |x|}$ for almost all $x\in\R$ (i.e., set
$c=1$). By \eqref{E4:SecMom-Lower}
and \eqref{E4:SecMom}, we may only consider the
case where $\rho(u) =\lambda u$.
Denote $J_0(t,x)= (e^{-\sd|\cdot|}*G_\kappa(t,\cdot))(x)$.
We first consider the case where $|x|\le \kappa t$.
As shown in Figure \ref{F4:ChangeVar},
split the integral that defines
$\left(J_0^2\star \calK\right)(t,x)$ over the three
regions I, II, and III, so that
\[\Norm{u(t,x)}_2^2\ge \left(J_0^2 \star
\calK\right)(t,x) =S_1 + S_2 +S_3
\ge S_3.\]
For arbitrary $a,b\in\;]0,1[$, we see that
\begin{align*}
S_3&\ge \frac{\lambda^2}{4}\int_{b t}^t \ud s
\int_{-a \kappa s}^{a \kappa s}\ud y\;
J_0^2\left(t-s,x-y\right) I_0\left(\sqrt{\frac{\lambda^2\left((\kappa s)^2
-y^2\right)}{2\kappa}}\right) \\
&\ge
\frac{\lambda^2}{4}\int_{b t}^t\ud s\;
I_0\left(\sqrt{\frac{\lambda^2\left(1-a^2\right)}{2\kappa}} \; \kappa s\right)
\int_{-a \kappa s}^{a \kappa s}\ud y \; J_0^2\left(t-s,x-y\right)\\
&\ge
\frac{\lambda^2}{4}
I_0\left(\sqrt{\frac{\lambda^2\left(1-a^2\right)}{2\kappa}} \; \kappa
b t\right)
\int_{b t}^t\ud s \int_{-a b \kappa t}^{a b\kappa t}\ud y\;
J_0^2\left(t-s,x-y\right).
\end{align*}
Clearly, for $\left(s,y\right)$ in Region III of Figure
\ref{F4:ChangeVar}, $|x-y|\le  \kappa (t-s)$ and so by Lemma
\ref{L:InitVel-Ind},
\[J_0\left(t-s,x-y\right) = \left(1-e^{-\sd \kappa (t-s)}
\cosh\left(\sd
(x-y)\right)\right)/\sd.\]
Using the inequalities $(a+b)^2\ge \frac{a^2}{2}-b^2$ and
$\cosh^2(x)=\frac{1}{2}\left(\cosh(2x)+1\right)\ge \frac{1}{2}\cosh(2x)$,
\[
J_0^2\left(t-s,x-y\right) \ge \frac{1}{4\sd^2}\:e^{-2\sd\kappa (t-s)}
\cosh(2\sd (x-y)) -\frac{1}{\sd^2}\:.
\]
Hence,
\[
\int_{b t}^t\ud s\int_{-a b \kappa t}^{a b\kappa t} \ud y\:
J_0^2\left(t-s,x-y\right) \ge
\frac{\left(1-e^{-2 (1-b)\sd  \kappa  t}\right) \cosh (2 \sd  x) \sinh (2 a b
\sd \kappa  t)}{8 \sd^4 \kappa }
-\frac{2 a (1-b) b \kappa  t^2}{\sd^2}
\:.
\]
Therefore, by \eqref{E4:AsyIn},
\begin{align}\label{E4_:InitVel}
\lim_{t\rightarrow+\infty}\frac{1}{t} \sup_{|x|\ge \alpha t}
\log\Norm{u(t,x)}_2^2 \ge 2\sd \alpha + 2 a b
\sd \kappa +b|\lambda| \sqrt{\kappa/2}\:  \sqrt{1-a^2}\; > 0,
\end{align}
for $\alpha \le \kappa$ and all $a,b\in\; ]0,1[\:$, which implies that
$\underline{\lambda}(2)\ge \kappa$.
As for the case where $|x|\ge \kappa t$, for all $a, b\in \;]0,1[$, by
Lemma \ref{L:InitVel-Ind},
\begin{align*}
\Norm{u(t,x)}_2^2 & \ge \left(J^2_0 \star\calK\right)(t,x)\\
&=\frac{\lambda^2}{16
\sd^2} \int_0^t\ud s \sinh^2(\sd
\kappa(t-s))
\int_{-\kappa s}^{\kappa s}\ud y \: e^{-2\sd |x-y|}
I_0\left(\sqrt{\frac{\lambda^2(\kappa^2s^2-y^2)}{2\kappa}}\right) \\
&\ge
\frac{\lambda^2 e^{-2\sd|x|+ 2 a \kappa b t \sd}}{32 \sd^3}
\left(
\frac{\sinh (2 (1-b) \sd  \kappa  t)}{4 \sd
   \kappa }-
\frac{1}{2} (1-b) t\right)
I_0\left(\sqrt{\frac{\lambda^2(1-a^2)}{2\kappa}} b\kappa t\right).
\end{align*}
Therefore, for $\alpha >\kappa$, we obtain the same inequality as
\eqref{E4_:InitVel}.
The rest argument is exactly the same as the proof of  part (2) of
Proposition \ref{P4:InitPos}. This completes the proof of Proposition
\ref{P4:InitVel}.
\end{proof}

\section{H\"older continuity in the stochastic wave equation}
\label{S:HolerWave}

\begin{theorem}\label{T4:Holder}
Suppose that $\rho$ is Lipschitz continuous.
If $g\in L_{loc}^{2\gamma}\left(\R\right)$, $\gamma\ge 1$ and $\mu \in
\calM\left(\R\right)$, then for all compact sets $K\in\R_+\times\R$ and all $p
\ge 1$, there is a constant $C_{K,p}$ such that for all $(t,x)$, $(t',x')\in K$,
\[
\Norm{I(t,x)-I(t',x')}_p\le C_{K,p}
\left(|t-t'|^{1/(2\gamma')}+|x-x'|^{1/(2\gamma')}\right),
\]
where $\frac{1}{\gamma}+ \frac{1}{\gamma'}=1$.
Hence,
\[I(t,x)\in
C_{\frac{1}{2\gamma'}-,\frac{1}{2\gamma'}-}\left(\R_+\times\R\right)\;\text{a.s.
}
\]
In addition, for all compact sets $K\in \R_+\times\R$ and $0\le
\alpha<1/(2\gamma')-2/p$,
\[
\E\left[\left(
\mathop{\sup_{(t,x),\;(s,y)\in K}}_{(t,x)\ne (s,y)}
\frac{|I(t,x)-I(s,y)|}{\left[|t-s|+|x-y|\right]
^\alpha } \right)^p\;\right ]<+\infty.
\]
In particular, if $g$ is locally bounded ($\gamma=+\infty$), then
$I(t,x)\in
C_{\frac{1}{2} -,\frac{1}{2}-}\left(\R_+\times\R\right)$ a.s.
\end{theorem}

\begin{proof}
We only need to verify that Assumption \ref{A:Holder} holds for
$K_n=[0,n]\times[-n,n]$. This is the case  thanks to
Propositions \ref{P4:H-bdd} -- \ref{P4:H-135} below.
More precisely,
let $J_{0,1}(t,x)$ and $J_{0,2}(t,x)$ be the homogeneous
solutions contributed respectively by $g$ and $\mu$. Clearly, when
both $g$ and $\mu$ are nonvanishing, $J_0(t,x) = J_{0,1}(t,x)+ J_{0,2}(t,x)$.
Because $J_0^2(t,x) \le 2 J_{0,1}^2(t,x)+ 2J_{0,2}^2(t,x)$,
we can consider $J_{0,1}(t,x)$ and $J_{0,2}(t,x)$ separately when verifying
Assumption \ref{A:Holder}.
In particular, Proposition \ref{P4:H-bdd} shows that the contribution of
$J_{0,2}(t,x)$ satisfies Assumption
\ref{A:Holder}, and Propositions
\ref{P4:H-246} and \ref{P4:H-135}  guarantee that the contribution of
$J_{0,1}(t,x)$ satisfies Assumption \ref{A:Holder}.
\end{proof}

\begin{proposition} \label{P4:H-Optimal}
Suppose that $\left|\rho(u)\right|^2=\lambda^2\left(\vv^2+u^2\right)$.
If $g(x)=|x|^{-a}$ with $a\in
\left[0,1/2 \right[$ and $\mu \equiv 0$, then
in the neighborhood of  the two characteristic lines $|x| = \kappa t$,
the function $I(t,x)$ mapping from $\R_+\times\R$ into $L^p(\Omega)$,
$p\ge 2$, cannot be
$\rho$-H\"older continuous either in space or in time with $\rho
>\frac{1-2a}{2}$.
\end{proposition}

This proposition is proved in Section \ref{SS4:H-Optimal}.

\begin{remark}[Optimal $L^p(\Omega)$-H\"older continuity]\label{R4:H-Optimal}
Clearly, $|x|^{-a}\in
L_{loc}^{2\gamma}\left(\R\right)$ if and only if $2 \gamma a<1$, i.e., $\gamma<
(2a)^{-1}$.
Hence, $\gamma'$,
the dual of $\gamma$, is strictly bigger than $(1-2a)^{-1}$.
Therefore, according to Theorem \ref{T4:Holder}, for all
$p\ge 2$, the function $I: \R_+\times\R \mapsto L^p(\Omega)$
is jointly $\eta$-H\"older continuous with $\eta = (1-2a)/2$.
For example, if $a=1/4$ (see Example \ref{Ex4:wave-g-1/4}), then $I$ is jointly
$1/4$-H\"older continuous in $L^p(\Omega)$.
Proposition \ref{P4:H-Optimal} then shows that $I(t,x)$ cannot be
jointly $\eta$-H\"older continuous with $\eta>1/4$. Hence, the estimates on the
joint $L^p(\Omega)$-H\"older continuity are optimal.
Singularities in the initial conditions affect the regularity of deviations from
the homogeneous solution. 
\end{remark}

\subsection{Three propositions for the H\"older continuity}
In this part, we will prove Propositions
\ref{P4:H-bdd} -- \ref{P4:H-135}, which together verify
Assumption \ref{A:Holder} (and hence the H\"older continuity).

\begin{proposition}\label{P4:G}
For $T>0$, we have that
\[\int_{\R_+} \ud s\int_\R  \ud y
\left(G_\kappa\left(t-s,x-y\right)-G_\kappa(t'-s,x'-y)\right)^2
\le  C_T \left(\left|x'-x \right| + \left|t'-t \right| \right),
\]
for all $(t,x)$ and $\left(t',x'\right)\in \:]0,T]\times\R$, with $C_T:=\left(
\kappa \vee 1\right)T/2$.
\end{proposition}
The proof of this proposition is elementary.

\begin{proposition}\label{P4:H-bdd}
Denote $K_n^*:= [0,n] \times [-n-\kappa n,n+\kappa n]$. Suppose that
\begin{align}\label{E:Wlbdd}
\sup_{(t,x)\in K_n^*} J_0^2(t,x)<+\infty,\quad\text{for all $n>0$.}
\end{align}
Then Assumption \ref{A:Holder} holds under the
settings: $\theta(t,x)\equiv 1$, $d=1$,
$\gamma_0=\gamma_1=1$, and $K_n=[0,n]\times[-n,n]$.
Condition \eqref{E:Wlbdd} (and hence Assumption \eqref{A:Holder})
holds in particular when  $g\equiv 0$ and $\mu$ is a
locally
finite
Borel measure:
\[\sup_{(t,x)\in K_n^*} J_0^2(t,x)
\le 1/4\; \Psi_\mu^* \left( n+2 \kappa n\right)<+\infty.\]
\end{proposition}
\begin{proof}
Fix $v\ge 0$, $n>1$ and choose arbitrary $(t,x)$ and $(t',x')\in
K_n=[0,n]\times[-n,n]$ (note that the time variable can be
zero).
Because the support of the function $\left(s,y\right)\mapsto
G_\kappa\left(t-s,x-y\right)-G_\kappa\left(t'-s,x'-y\right)$
is included in the compact set $K_n^*$, by Proposition \ref{P4:G}, the l.h.s. of
\eqref{E:H-135} is bounded by, 
\begin{align*}
C_n
\iint_{\R_+\times\R} \ud s\ud y\:
\left(G_\kappa\left(t-s,x-y\right)-G_\kappa\left(t'-s,x'-y\right)\right)^2
\le C_n \frac{n\left(\kappa\vee 1\right)}{2}  \left(\left|x-x'\right|+
\left|t-t'\right|\right),
\end{align*}
where
$C_n=\sup_{\left(s,y\right)\in K_n^*}
\left(v^2+2J_0^2\left(s,y\right)\right)$.
As for \eqref{E:H-246}, using the same constant $C_n$, the l.h.s. of
\eqref{E:H-246} is bounded by
\begin{multline*}
 C_n
\iint_{\R_+\times\R}\ud s\ud y  \left[\iint_{\R_+\times\R}\ud u\ud z\:
G_\kappa^2(s-u,y-z)\right]
\left(G_\kappa\left(t-s,x-y\right)-G_\kappa\left(t'-s,x'-y\right)\right)^2
\\
\le
\frac{C_n\kappa n^2}{4}
\iint_{\R_+\times\R}\ud s\ud y\:
(G_\kappa\left(t-s,x-y\right) -G_\kappa\left(t'-s,x'-y\right))^2.
\end{multline*}
Then apply Proposition \ref{P4:G} as before.
\end{proof}

\begin{proposition}\label{P4:H-246}
Suppose $\mu\equiv 0$ and $g\in L_{loc}^2\left(\R\right)$. Then
\eqref{E:H-246} holds
with $\theta(t,x)\equiv 1$, $d=1$, $\gamma_0=\gamma_1=1$, and
$K_n=[0,n]\times[-n,n]$.
\end{proposition}
\begin{proof}
Split \eqref{E:H-246} into two parts by
linearity: one term is contributed
by $v^2$ and the other by $2J_0^2$. Proposition \ref{P4:H-bdd} shows that the
first term satisfies Assumption \ref{A:Holder}.
Hence, we only need to consider the second term.
Let $K_n^*=[0,n]\times \left[-(1+\kappa)n,(1+\kappa)n\right]$.
By a change of variables (see Figure \ref{F4:ChangeVar}),
for all $(t,x)\in K_n^*$,
\[
\left(J_0^2\star G_\kappa^2\right)(t,x)
= \frac{1}{16} \frac{1}{2\kappa} \iint_{I \cup II \cup III}\ud u \ud w\:
\left(g(w)+g(u)\right)^2
\le \frac{( 1+\kappa) n}{4\kappa}
\Psi_{g}(n+n\kappa),
\]
where $I$, $II$ and $III$ denote the three domains shown in Figure
\ref{F4:ChangeVar}.
Therefore, this proposition is proved by applying Proposition
\ref{P4:H-bdd}.
\end{proof}

\begin{proposition}\label{P4:H-135}
 Suppose $\mu\equiv 0$, $g\in L_{loc}^{2\gamma}\left(\R\right)$ with $\gamma\ge
1$, and
$1/\gamma+1/\gamma'=1$. Then \eqref{E:H-135} holds
with $\theta(t,x)\equiv 1$, $d=1$, and $\gamma_0=\gamma_1=1/\gamma'$.
\end{proposition}
\begin{proof}
Equivalently, we shall show that
\eqref{E:H-1}--\eqref{E:H-5} hold under the same settings.
As explained in the proof of Proposition \ref{P4:H-246}, we can assume that
$v=0$ in \eqref{E:H-1}--\eqref{E:H-5}.
Fix $n>0$, $(t,x)$ and $(t',x')\in K_n=[0,n]\times[-n,n]$ with $t\le t'$.
We first prove \eqref{E:H-1}. Because the
support of the function $G_\kappa-G_\kappa$ is in
$K_n^*=[0,n]\times[-(1+\kappa)n,(1+\kappa)n]$, by H\"older's inequality,
\begin{align*}
I&:= \int_0^t \ud s \int_\R\ud y\:
J_0^2\left(s,y\right)\left(G_\kappa\left(t-s,x-y\right)-G_\kappa(t'-s,
x-y)\right)^2  \\
&\le \int_0^t \ud s \left(\int_{-(1+\kappa)n}^{(1+\kappa)n}\ud y\:
J_0^{2\gamma}\left(s,y\right) \right)^{1/\gamma}\left(
\int_\R\ud y
\left|G_\kappa\left(t-s,x-y\right)-G_\kappa(t'-s,x-y)\right|^{2\gamma'}
\right)^{1/\gamma'}.
\end{align*}
By convexity of $x\mapsto |x|^{2\gamma}$,
\[
\int_{-(1+\kappa)n}^{(1+\kappa)n}\ud y\:
J_0^{2\gamma}\left(s,y\right)   \le
\frac{1}{2} \int_{-(1+\kappa)n}^{(1+\kappa)n}\ud y\:
\left( g^{2\gamma}(y+\kappa s)+ g^{2\gamma}(y-\kappa s) \right)
\le \Psi_{g^\gamma}(n+2\kappa n).
\]
Hence,
\[I\le \Psi_{g^\gamma}^{\frac{1}{\gamma}}(n+2\kappa n)
\int_0^t \ud s \left(
\int_\R\ud y
\left|G_\kappa\left(t-s,x-y\right)-G_\kappa(t'-s,x-y)\right|^{2\gamma'}
\right)^{1/\gamma'},\]
where
\[\int_\R\ud y\:
\left|G_\kappa\left(t-s,x-y\right)-G_\kappa(t'-s,x-y)\right|^{2\gamma'}
=
2^{-2\gamma'} \kappa n \left|t'-t \right|.\]
Therefore, \[I\le\frac{\kappa^{1/\gamma'} n^{1+1/\gamma'}}{4}
\Psi_{g^\gamma}^{\frac{1}{\gamma}}(n+2\kappa n) \left|t'-t
\right|^{1/\gamma'},\] which proves
\eqref{E:H-1}.

Now let us consider \eqref{E:H-3}. As above, we can assume that $v=0$, so we set
\begin{align*}
I&:= \int_0^t \ud s \int_\R \ud y\:
J_0^2\left(s,y\right)(G_\kappa\left(t-s,x-y\right)-G_\kappa(t-s,
x'-y))^2 \\
& \le
\Psi_{g^\gamma}^{\frac{1}{\gamma}}(n+2\kappa n)
\int_0^t \ud s \left(
\int_\R  \ud y\:
\left|G_\kappa\left(t-s,x-y\right)-G_\kappa(t-s,x'-y)\right|^{2\gamma'}
\right)^{1/\gamma'},
\end{align*}
where (see Figure \ref{F4:WaveKernel}),
\begin{multline*}
\int_\R\ud y \:
\left|G_\kappa\left(t-s,x-y\right)-G_\kappa(t-s,x'-y)\right|^{ 2\gamma' }
\\ =
2^{1-2\gamma'}\left|x'-x \right| \: \Indt{\left|x'-x \right|\le 2\kappa (t-s)}
+
2^{1-2\gamma'} \kappa (t-s)
\: \Indt{\left|x'-x \right|> 2\kappa (t-s)}
\le
2^{1-2\gamma'}\left|x'-x \right|.
\end{multline*}
Therefore,
\[I\le 2^{-2+1/\gamma'} n\:   \Psi_{g^\gamma}^{\frac{1}{\gamma}}(n+2\kappa n) \:
\left|x'-x
\right|^{1/\gamma'},\]
which proves \eqref{E:H-3}.

Now let us consider \eqref{E:H-5}. By the same
arguments as above, we only consider
\begin{align*}
I  &:= \int_t^{t'} \ud s \int_\R\ud y\:
J_0^2\left(s,y\right)G_\kappa^2(t'-s,x'-y)\\
&\:\le
\Psi_{g^\gamma}^{\frac{1}{\gamma}}(n+2\kappa n)
\int_t^{t'} \ud s \left(
\int_\R\ud y\: G_\kappa^{2\gamma'}(t'-s,x'-y)
\right)^{1/\gamma'},
\end{align*}
where
\[\int_\R \ud y\: G_\kappa^{2\gamma'}(t'-s,x'-y)
= 2^{-2\gamma'} 2 \kappa (t'-s) \le 2^{-2\gamma'} 2 \kappa n.\]
Therefore,
\[I\le 2^{-2+1/\gamma'} (n \kappa)^{1/\gamma'}
 \Psi_{g^\gamma}^{\frac{1}{\gamma}}(n+2\kappa n) \left|t'-t \right|.\]
Finally, \eqref{E:H-5} follows from the bound $\left|t'-t \right| \le
n^{1/\gamma} \left|t'-t \right|^{1/\gamma'}$.
\end{proof}

\subsection{Optimality of the H\"older exponents (proof of Proposition
\ref{P4:H-Optimal})}
\label{SS4:H-Optimal}
\begin{lemma}\label{L:Special-g}
If $g(x)=|x|^{-a}$ with $a\in
\left[0,1/2 \right[$ and $\mu \equiv 0$, then
\[
\left(J_0^2\star G_\kappa^2\right)(t,x)
=
\begin{cases}
\frac{a^2-4a+2}{32\kappa (1-2a)(1-a)^2} \left|\kappa t-x\right|^{2(1-a)},&
\text{if $x< -\kappa t$,}\cr
\frac{1}{32\kappa (1-a)^2} \left[\left(\kappa t-x\right)^{1-a}+\left(\kappa
t+x\right)^{1-a}\right]^2 & \cr
\quad+
\frac{t}{16(1-2a)}\left[\left(\kappa t-x\right)^{1-2a}
+\left(\kappa t+x\right)^{1-2a}\right],& \text{if $|x|\le \kappa
t$,}\cr
\frac{a^2-4a+2}{32\kappa (1-2a)(1-a)^2} \left|\kappa t+x\right|^{2(1-a)},&
\text{if $x> \kappa t$,}
\end{cases}
\]
where $J_0(t,x)= \left( g\left(x-\kappa t\right) + g\left(x+\kappa
t\right)\right)/2$.
\end{lemma}

\begin{proof}
First assume that $|x|\le \kappa t$. Then
\[\left(J_0^2\star G_\kappa^2\right)(t,x) = \frac{1}{16}\int_0^t \ud s
\int_{x-\kappa(t-s)}^{x+\kappa (t-s)} \ud y\:(g(y-\kappa s) +
g(y+\kappa
s) )^2  = \frac{1}{16}\left(S_1 + S_2+S_3\right),\]
where $S_1$, $S_2$ and $S_3$ correspond to the integrations in the regions I,
II and III shown in Figure \ref{F4:ChangeVar}.
To evaluate these three integrals, by change the variables (see Figure
\ref{F4:ChangeVar}),
\begin{gather*}
 S_1 =  \frac{1}{2\kappa}\int_{x-\kappa t}^0 \ud w
\int_{-w}^{-x+\kappa t}\ud u
\left(\left|u\right|^{-a}+\left|w\right|^{-a}\right)^2=
\frac{a^2-4a+2}{2\kappa(1-2a)(1-a)^2} \left(\kappa t-x\right)^{2(1-a)},\\
S_2=  \frac{1}{2\kappa}\int_0^{x+\kappa t}\ud w \int_{-w}^0
\ud u\left(\left|u\right|^{-a}+\left|w\right|^{-a}\right)^2=
\frac{a^2-4a+2}{2\kappa(1-2a)(1-a)^2} \left(\kappa t+x\right)^{2(1-a)},\\
S_3= \frac{1}{\kappa (1-a)^2}\left(\kappa^2 t^2-x^2\right)^{1-a} +
\frac{1}{2\kappa (1-2a)}\left(
\left(\kappa t-x\right)^{1-2a}\left(\kappa t +x\right)+\left(\kappa
t+x\right)^{1-2a}\left(\kappa t-x\right)
\right).
\end{gather*}
Use the fact that
\[\frac{a^2-4a+2}{2\kappa(1-2a)(1-a)^2} =
\frac{(1-2a)+(1-a)^2}{2\kappa(1-2a)(1-a)^2}
=\frac{1}{2\kappa (1-a)^2} + \frac{1}{2\kappa (1-2a)}\]
to sum up these $S_i$.
The other two cases,
$x<-\kappa t$ and $x>\kappa t$, can be calculated similarly to $S_1$ and $S_2$
respectively.
\end{proof}

\begin{proof}[Proof of Proposition \ref{P4:H-Optimal}]
Let $I(t,x)$ be the stochastic integral part of random field solution, i.e.,
$u(t,x)=J_0(t,x)+I(t,x)$. For $(t,x)$ and $(t',x')\in\R_+\times\R$, because
\[
\vv^2+\Norm{u\left(s,y\right)}_2^2 \ge J_0^2\left(s,y\right),\quad\text{and}
\quad
\Norm{I(t,x)-I(t',x')}_p^2 \ge
\Norm{I(t,x)-I(t',x')}_2^2
\]
for $p\ge 2$, we see that
\begin{multline}
\Norm{I(t,x)-I(t',x')}_p^2
\\
\ge\lambda^2 \iint_{\R_+\times\R}\ud s \ud y\:
\left(G_\kappa\left(t-s,x-y\right)-G_\kappa(t'-s,x'-y)\right)^2
J_0^2\left(s,y\right).\quad
\label{E4:H-Opt-Dif}
\end{multline}

{\noindent\em Spatial increments.}
Fix $t=t'>0$, $x$ and $x'\in\R$. Denote $T=T_\kappa(t,x-x')$.
By \eqref{E4:GG}, the lower bound in \eqref{E4:H-Opt-Dif} reduces to
\[\lambda^2 \iint_{\R_+\times\R}\ud s \ud y\:
J_0^2\left(s,y\right)
(
G_\kappa^2\left(t-s,x-y\right) -
2 G_\kappa^2\left(T-s,\frac{x+x'}{2}-y\right)+
G_\kappa^2(t-s,x'-y)
),\] which is denoted by $\lambda^2 L(t,x,x')$. Then
\[L(t,x,x')
=
\left(J_0^2 \star G_\kappa^2\right)(t,x)
+\left(J_0^2 \star G_\kappa^2\right)(t,x')
-2 \left(J_0^2 \star
G_\kappa^2\right)\left(T,\frac{x+x'}{2}\right).\]
Let $x=\kappa t$ and $x'<x$ be such that $|x'-x|\le 2\kappa t$. Hence,
$T_\kappa(t,x-x') = t-(x-x')/(2\kappa)$.
Then apply Lemma \ref{L:Special-g} to see that
\[L(t,\kappa t,x')
= \frac{1}{32\kappa(1-a)^2}
L_1(t,x')
+ \frac{t}{16(1-2a)} L_2(t,x'),\]
with
\begin{align*}
 L_1(t,x') &= (2\kappa t)^{2(1-a)} + \left[\left(\kappa
t-x'\right)^{1-a}+\left(\kappa
t+x'\right)^{1-a}\right]^2 -
2 \left(\kappa
t+x'\right)^{2(1-a)},\\
L_2(t,x') &=
(2\kappa t)^{1-2a} + (\kappa t-x')^{1-2a} - (\kappa
t+x')^{1-2a}.
\end{align*}
Let $h=\kappa t-x'$. Then
\begin{align*}
L_1(t,x') &= (2\kappa t)^{2(1-a)} + \left[h^{1-a}+(2\kappa t-h)^{1-a}\right]^2
-2 \left(2\kappa t -h\right)^{2(1-a)}
\ge h^{2(1-a)},\\
L_2(t,x') &= (2\kappa t)^{1-2a} + h^{1-2a}- \left(2\kappa t -h\right)^{1-2a}
\ge h^{1-2a}.
\end{align*}
Since $1-2a  \in \:\left]0,1\right]$ and $2(1-a) \in \;\left]1,2\right]$, by
discarding $L_1(t,x')$, we have that
\[\Norm{I(t,\kappa t)-I(t,\kappa t-h)}_p^2
= \lambda^2 L(t,\kappa t,x') \ge \frac{\lambda^2 t}{16(1-2a)}
h^{1-2a}.\]

{\noindent\em Time increments.}
Now fix $x=x'\in\R$. By symmetry, we assume that  $x>0$ . For $t'\ge t\ge 0$,
 \eqref{E4:H-Opt-Dif} implies that
\[\Norm{I(t,x)-I(t',x)}_p^2\ge
\lambda^2
\left(
\left(J_0^2 \star G_\kappa^2\right)(t',x)
-\left(J_0^2 \star G_\kappa^2\right)(t,x)
\right),\]
because $G_\kappa(t,x)
G_\kappa(t',x) = G_\kappa^2(t,x)$.
Take $t=x/\kappa$ and $h=t'-t = t'-x/\kappa$. Similarly to the
previous case,
\[
\left(J_0^2 \star G_\kappa^2\right)\left(\frac{x}{\kappa},x\right) =
\frac{1}{32\kappa(1-a)^2}\left(2x\right)^{2(1-a)}
+ \frac{x}{16\kappa (1-2a)} \left(2x\right)^{1-2a},
\]
and $\left(J_0^2 \star
G_\kappa^2\right)\left(t',x\right)$ is equal to
\[ \frac{1}{32\kappa (1-a)^2}\left[
\left(\kappa h\right)^{1-a}+\left(\kappa
h+2x\right)^{1-a}\right]^2
+\frac{x}{16\kappa(1-2a)} \left[\left(\kappa h\right)^{1-2a}+\left(\kappa
h+2x\right)^{1-2a}\right].\]
Hence, by symmetry, for all $x\in\R$, and $h=t'-|x|/ \kappa>0$,
\[\Norm{I\left(\frac{|x|}{\kappa},x\right)-I(t',x)}_p^2 \ge
\frac{\lambda^2 |x| }{16 \kappa^{2a} (1-2a)}
h^{1-2a}.\]
Therefore, Proposition \ref{P4:H-Optimal} is proved.
\end{proof}

\appendix
\section{Some technical lemmas}
\begin{lemma}\label{L:IntShCh}
For $a\ne 0$ and $t\ge 0$,
$\int_0^t\ud s\: \cosh(a s) (t-s) = a^{-2}\left(\cosh(a
t)-1\right)$,
$\int_0^t \ud s\: \sinh(a s) (t-s) = a^{-2}\left(\sinh(a t)-a t\right)$,
and
$\int_0^t \ud s\: \sinh(a s) (t-s)^2 = a^{-3}\left(2\cosh(a t)-a^2
t^2-2\right)$.
\end{lemma}

\begin{lemma} \label{L:Ht}
For $t\ge 0$ and $x\in\R$, we have that
$\int_\R \ud x \:\calK(t,x) =
|\lambda|\:(\kappa/2)^{1/2}\sinh\left(|\lambda| \: (\kappa/2)^{1/2} t\right)$
and
$\left(1\star \calK\right)(t,x) =
\cosh\left(|\lambda| \: (\kappa/2)^{1/2}
t\right) -1 $.
\end{lemma}
\begin{proof}
By a change of variable,
\[
\int_\R \ud x\: \calK(t,x) = 2 \int_0^{|\lambda|\sqrt{\kappa/2}\: t}\ud y\:
\frac{\lambda^2}{4}
\frac{\sqrt{2\kappa}}{|\lambda|} \frac{y}{\sqrt{\kappa t^2\lambda^2/2-y^2}}
I_0(y). 
\]
Then the first statement follows from \cite[(6) on
p. 365]{Erdelyi1954-II} with $\nu=0$, $\sigma=1/2$ and
$a=|\lambda|\: (\kappa/2)^{1/2}t$.  The second statement  is a simple
application
of the first.
\end{proof}

\begin{lemma}
\label{L:coshsinh}
Suppose that $a\ne c$, $t>0$ and $b\in [0,1]$. Then
\begin{multline*}
\int_{b t}^t \ud s\:\cosh\left(a(t-s)\right)\sinh\left(c s\right)
\\
=
\left(a^2-c^2\right)^{-1}\Big(
c \cosh(b c t)\cosh\left(a(1-b)t\right)
-c \cosh( c t) +a\sinh(bct)\sinh\left(a(1-b)t\right)
\Big).
\end{multline*}
\end{lemma}
\begin{proof}
Use the formula
$\cosh(x)\sinh(y)=\frac{1}{2}\left(\sinh(x+y)+\sinh(-x+y)\right)$.
\end{proof}

For the following two lemmas, let $G_\nu(t,x)$, $\nu>0$, be the
heat kernel function (see \eqref{E:HeatG}).
\begin{lemma}\label{L2:GG}
For all $t$, $s>0$ and $x$, $y\in\R$, we have that $ G_\nu^2(t,x) =
\frac{1}{\sqrt{4\pi\nu t}} G_{\nu/2}(t,x)$ and $G_\nu(t,x)G_\nu\left(s,y\right)
= G_\nu \left(\frac{ts}{t+s},\frac{s x+t
y}{t+s}\right)
G_\nu\left(t+s,x-y\right)$.
\end{lemma}

\begin{lemma}[Lemma 4.4 of
\cite{ChenDalang13Heat}]\label{LH:Split}
 For all $x$, $z_1$ $z_2\in\R$ and $t,s>0$, denote
$\bar{z} = \frac{z_1+z_2}{2}$, $\Delta z = z_1-z_2$. Then
$G_1\left(t,x-\bar{z}\right)
G_1\left(s,\Delta z\right)
\le \frac{(4t) \vee s}{\sqrt{t s}}
G_1\!\left((4t)\vee s,x-z_1\right)
G_1\!\left((4t)\vee s,x-z_2\right)$, where $a\vee
b:=\max(a,b)$.
\end{lemma}

\def\polhk#1{\setbox0=\hbox{#1}{\ooalign{\hidewidth
  \lower1.5ex\hbox{`}\hidewidth\crcr\unhbox0}}} \def\cprime{$'$}


\end{document}